\documentclass[a4paper, 11pt]{amsart}
\usepackage{a4wide}
\usepackage[english]{babel} 
\usepackage{amsmath, amssymb, bbm} 
\usepackage[all,cmtip]{xy} 
\usepackage{tikz}
\usepackage{tikz-cd}
\usepackage{xcolor}
\usepackage[normalem]{ulem}
\usepackage[mathscr]{euscript}
\usepackage{pifont}
\usepackage{hyperref}
\usepackage[capitalize,noabbrev]{cleveref}
\usepackage{dsfont}
\usepackage{comment}

\usepackage{pdftexcmds}

\DeclareFontFamily{U}{cbgreek}{}
\DeclareFontShape{U}{cbgreek}{m}{n}{
        <-6>    grmn0500
        <6-7>   grmn0600
        <7-8>   grmn0700
        <8-9>   grmn0800
        <9-10>  grmn0900
        <10-12> grmn1000
        <12-17> grmn1200
        <17->   grmn1728
      }{}
\DeclareFontShape{U}{cbgreek}{bx}{n}{
        <-6>    grxn0500
        <6-7>   grxn0600
        <7-8>   grxn0700
        <8-9>   grxn0800
        <9-10>  grxn0900
        <10-12> grxn1000
        <12-17> grxn1200
        <17->   grxn1728
      }{}

\DeclareRobustCommand{\Qoppa}{%
  \text{\usefont{U}{cbgreek}{\normalorbold}{n}\symbol{21}}%
}
\makeatletter
\newcommand{\normalorbold}{%
  \ifnum\pdf@strcmp{\math@version}{bold}=\z@ bx\else m\fi
}
\makeatother

\newtheoremstyle{thms}
	{}{}{\itshape}{}{\bfseries }{}{ }
	{\thmname{#1} \thmnumber{#2}. \thmnote{\bfseries{[#3]}}}
\newtheoremstyle{name}
	{}{}{\itshape}{}{\bfseries }{}{ }
	{\thmname{#1}\thmnumber{#2}\thmnote{\bfseries{[#3]}}}
\newtheoremstyle{defs}
	{}{}{\normalfont}{}{\bfseries }{}{ }
	{\thmname{#1} \thmnumber{#2}. \thmnote{\bfseries{(#3)}}}
\newtheoremstyle{thms2}
	{}{}{\itshape}{}{\bfseries }{}{ }
	{\thmname{#1}\thmnumber{#2}. \thmnote{\bfseries{[#3]}}}
\newtheoremstyle{rmk}
	{}{}{\normalfont}{}{\bfseries }{}{ }
	{\thmname{#1} \thmnumber{#2}. \thmnote{\bfseries{(#3)}}}
\newtheoremstyle{claim}
	{}{}{\normalfont}{}{\itshape}{}{ }{\thmname{#1} \thmnumber{#2}. \thmnote{#3}}

\newtheorem{Prop}{Proposition}[section]
\newtheorem{Thm}[Prop]{Theorem}
\newtheorem{add}[Prop]{Addendum}
\newtheorem{Lemma}[Prop]{Lemma}
\newtheorem{Cor}[Prop]{Corollary}

\newtheorem*{conj}{Question}

\theoremstyle{defs}
\newtheorem{Def}[Prop]{Definition}
\newtheorem{Example}[Prop]{Example}

\theoremstyle{thms2}

\theoremstyle{rmk}
\newtheorem{Rmk}[Prop]{Remark}

\theoremstyle{plain}
\newcounter{zaehler}

\newtheorem{introthm}[zaehler]{Theorem}

\newtheorem*{introcor*}{Corollary}

\theoremstyle{claim}

\DeclareMathOperator*{\colim}{colim}
\DeclareMathOperator{\fib}{fib}
\newcommand{\C}{\mathscr{C}}

\newcommand{\E}{\mathbb{E}}
\newcommand{\R}{\mathbb{R}}

\newcommand{\Sp}{\mathrm{Sp}}
\newcommand{\Ab}{\mathrm{Ab}}
\newcommand{\Hom}{\mathrm{Hom}}
\newcommand{\Ext}{\mathrm{Ext}}
\newcommand{\Z}{\mathbb{Z}}
\newcommand{\Q}{\mathbb{Q}}
\newcommand{\map}{\mathrm{map}}
\newcommand{\Map}{\mathrm{Hom}}

\newcommand{\lto}{\longrightarrow}
\renewcommand{\H}{\mathrm{H}}
\newcommand{\KU}{\mathrm{KU}}
\newcommand{\KO}{\mathrm{KO}}

\newcommand{\bC}{\mathbb{C}}
\renewcommand{\L}{\mathrm{L}}

\newcommand{\F}{\mathbb{F}}
\renewcommand{\S}{\mathbb{S}}
\newcommand{\dR}{\mathrm{dR}}
\newcommand{\MSG}{\mathrm{MSG}}

\newcommand{\MSO}{\mathrm{MSO}}
\newcommand{\BGl}{\mathrm{BGl}}
\newcommand{\BSl}{\mathrm{BSl}}
\newcommand{\Sl}{\mathrm{Sl}}

\newcommand{\BSO}{\mathrm{BSO}}
\newcommand{\BSG}{\mathrm{BS}\G}
\newcommand{\B}{\mathrm{B}}
\newcommand{\I}{\mathrm{I}}
\newcommand{\G}{\mathcal{G}}
\newcommand{\Top}{\mathrm{Top}}

\newcommand{\id}{\mathrm{id}}
\newcommand{\can}{\mathrm{can}}
\newcommand{\LLL}{\mathscr{L}}

\newcommand{\QF}{\Qoppa}
\newcommand{\Catp}{{\mathrm{Cat}_\infty^\mathrm{p}}}
\newcommand{\Dperf}{{\mathscr{D}^{\mathrm{perf}}}}
\newcommand{\Proj}{\mathrm{Proj}}
\newcommand{\gl}{g\Lambda}

\newcommand{\Ringinv}{\mathrm{Ring}^{\mathrm{inv}}}
\newcommand{\gs}{\mathrm{gs}}
\renewcommand{\G}{\mathrm{G}}
\newcommand{\gq}{\mathrm{gq}}
\newcommand{\q}{\mathrm{q}}
\newcommand{\s}{\mathrm{s}}
\newcommand{\n}{\mathrm{n}}

\begin{document}

\title{On the homotopy type of L-spectra of the integers}
\date{\today}

\author[F.~Hebestreit]{Fabian Hebestreit}
\address{Mathematisches Institut, RFWU Bonn, Germany}
\email{f.hebestreit@math.uni-bonn.de}

\author[M.~Land]{Markus Land}
\address{Department of Mathematical Sciences, University of Copenhagen, Denmark}
\email{markus.land@math.ku.dk}

\author[T.~Nikolaus]{Thomas Nikolaus}
\address{Mathematisches Institut, WWU M\"unster, Germany}
\email{nikolaus@uni-muenster.de}

\begin{abstract}
We show that quadratic and symmetric $\L$-theory of the integers are related by Anderson duality and that both spectra split integrally into the $\L$-theory of the real numbers and a generalised Eilenberg-Mac Lane spectrum. As a consequence, we obtain a corresponding splitting of the space $\G/\Top$. Finally, we prove analogous results for the genuine L-spectra recently devised for the study of Grothendieck--Witt theory.
\end{abstract}

\maketitle

\tableofcontents

\section{Introduction}

The goal of the present article is to investigate homotopy theoretic properties of  L-spectra of the integers. We will concentrate on two particular flavours: On the one hand, we shall consider the classical quadratic L-spectrum $\L^{\q}(\Z)$, whose homotopy groups arise as receptacles for surgery obstructions in geometric topology, along with its symmetric companion $\L^{\s}(\Z)$ as introduced by Ranicki. On the other, we will study the genuine L-spectra $\L^{\gq}(\Z)$ and $\L^\gs(\Z)$, that more intimately relate to classical Witt-groups of unimodular forms and are part of joint work with Calm\`es, Dotto, Harpaz, Moi, Nardin and Steimle on the homotopy type of Grothendieck--Witt spectra. We will treat each case in turn.

\subsection{The classical variants}
As indicated above, the main motivation for studying the classical L-groups comes from their relation to the classification of manifolds as established by Wall: One of the main results of (topological) surgery theory associates to an $n$-dimensional closed, connected, and oriented topological manifold $M$ the \emph{surgery exact sequence}
\[ \dots \lto \mathcal{N}_\partial(M\times I) \stackrel{\sigma}{\lto} \L_{n+1}^\q(\Z\pi_1(M)) \lto \mathcal{S}(M) \lto \mathcal{N}(M) \stackrel{\sigma}{\lto} \L_n^\q(\Z\pi_1(M)). \]
Here, we have denoted by $\mathcal S(M)$, the main object of interest, the set of $h$-cobordism classes of homotopy equivalences to $M$ and by $\mathcal N(M)$ the set of cobordism classes of degree $1$ normal maps to $M$. Finally, $\L_n^\q(\Z\pi_1(M))$ consists of bordism classes of $n$-dimensional quadratic Poincar\'e chain complexes. The classification of manifolds homotopy equivalent to $M$ therefore boils down to understanding the \emph{surgery obstruction} $\sigma$, assigning to a degree $1$ normal map $f \colon N \rightarrow M$ its surgery kernel $\sigma(f) = \fib\big(\mathrm{C}^*(M) \rightarrow \mathrm{C}^*(N)\big)[1]$ with the Poincar\'e structure induced from the Poincar\'e duality of $M$ and $N$.

Now, celebrated work of Ranicki and Sullivan shows that $\sigma$ can be described entirely in terms of L-spectra: It is identified with the composite
\[ (\tau_{\geq 1}\L^{\q}(\Z))_*(M) \lto \L^{\q}(\Z)_*(M) \stackrel{\mathrm{asbl}}{\lto} \L^{\q}_{*}(\Z\pi_1(M)),\]
where $\tau_{\geq 1}$ denotes connected cover of a spectrum, and $\mathrm{asbl}$ is the $\L$-theoretic assembly map. As an application of this technology let us mention recent progress on Borel's famous rigidity conjecture, which states that any homotopy equivalence between aspherical manifolds is homotopic to a homeomorphism. {In dimensions greater than $4$ the topological $s$-cobordism theorem translates the Borel conjecture to the statement that $\mathcal S^s(M) = \{\id_M\}$, whenever $M$ is aspherical;  here $\mathcal S^s(M)$ denotes the set of $s$-cobordism classes of simple homotopy equivalences to $M$. That the assembly map in L-theory is an isomorphism implies that $\mathcal{S}(M) = \{\id_M\}$ and that a similar assembly map in K-theory is an isomorphism yields $\mathcal S^s(M) = \mathcal S(M)$. The combined statement about the assembly maps is known as the Farrell--Jones conjecture, which has been attacked with remarkable success, leading to a proof of Borel's conjecture for large classes of aspherical manifolds; see for example \cite{FJ3,BL2,BLRR}.}\\

It is therefore of great interest to investigate L-spectra, and not just their homotopy groups. Our first set of results relates the symmetric and quadratic $\L$-spectra of the integers, $\L^{\s}(\Z)$ and $\L^{\q}(\Z)$ to the $\L$-spectrum $\L(\R)$ of the real numbers; since symmetric bilinear forms over the real numbers admit unique quadratic refinements, there is no difference between the quadratic and symmetric $\L$-spectra of $\R$, as reflected in our notation. To state our result we note, finally, that by work of Laures and McClure \cite{LauresMCCII} the tensor product of symmetric forms induces an $\E_\infty$-ring structure on symmetric L-spectra, over which the quadratic variants form modules. We show:

\begin{introthm}\label{thmA}
As $\L^{\s}(\Z)$-module spectra, $\L^{\q}(\Z)$ and $\L^{\s}(\Z)$ are Anderson dual to each other. Furthermore, there is a canonical $\E_1$-ring map $\L(\R) \to \L^{\s}(\Z)$ splitting the $\E_\infty$-ring map $\L^{\s}(\Z) \rightarrow \L(\R)$ induced by the homomorphism $\Z \rightarrow \R$. Using it to regard $\L^{\s}(\Z)$ and $\L^{\q}(\Z)$ as $\L(\R)$-modules there are then canonical equivalences
\[
\L^{\s}(\Z)  \simeq \L(\R) \oplus (\L(\R)/2)[1] \quad \text{and}\quad
\L^{\q}(\Z)  \simeq \L(\R) \oplus (\L(\R)/2)[-2] \]
of $\L(\R)$-modules. Finally, as a plain spectrum, $\L(\R)/2 \simeq \bigoplus_{k \in \mathbb Z} \H\Z/2[4k]$.
\end{introthm}

Some comments are in order:
\begin{enumerate}
\item As we recall in \cref{sectiona}, Anderson duality is a functor $\I \colon \Sp^{\mathrm{op}} \rightarrow \Sp$ on the category of spectra which is designed to admit universal coefficient sequences of the form 
\[0 \lto \Ext(E_{n-1}(Y),\Z) \lto \I(E)^n(Y) \lto \Hom(E_n(Y),\Z) \lto 0\]
for any two spectra $E$ and $Y$. Thus, the first statement of \cref{thmA} gives a useful relation between $\L^{\q}(\Z)$-cohomology and $\L^{\s}(\Z)$-homology, and vice versa.
\item The homotopy groups of the $\L$-theory spectra occurring in the theorem are of course well-known by work of Milnor and Kervaire for the quadratic case, and Mishchenko and Ranicki for the symmetric one.%, namely
%\[\L_\ast^\s(\mathbb Z) = \begin{cases} 
%0 & i = 4k+2 \\
%0 & i=4k+3 \\
%\mathbb Z/2 &  i=4k+1 \\
%\mathbb Z & i=4k \end{cases} \quad \text{and} \quad 
%\L_\ast^\s(\mathbb Z) = \begin{cases} 
%0 & i = 4k+2 \\
%\mathbb Z/2 & i=4k+3 \\
%0 &  i=4k+1 \\
%\mathbb Z & i=4k \end{cases}.\]

Furthermore, Sullivan showed, that after inverting $2$, they are all equivalent to the real $K$-theory spectrum $\KO$, albeit with a non-canonical equivalence; see \cite{Sullivan, MST}. It is also well-known, that after localization at 2, $\L$-spectra become generalised Eilenberg-Mac Lane spectra; see \cite{TW}. As a result,  the splittings in \cref{thmA} are fairly easy to deduce at the level of underlying spectra.  
\item Refining the results of the previous point, the homotopy type of $\L(\R)$ is quite well understood, even multiplicatively: For instance, by \cite{KandL2} there is a canonical map of $\E_\infty$-rings $\mathrm{ko} \rightarrow \L(\mathbb R)$, that induces an equivalence $\tau \colon \KO\mathopen{}\left[\tfrac 1 2\right]\mathclose{} \rightarrow \L(\R)\mathopen{}\left[\tfrac 1 2\right]\mathclose{}$ and (at least) as an $\E_1$-ring $\L(\R)_{(2)}$ is a free $\H\Z_{(2)}$-algebra on an invertible generator $t$ of degree $4$ as a consequence of the Hopkins-Mahowald theorem, which we will review in the body the paper. Its $2$-local fracture square thus displays the $\E_1$-ring underlying $\L(\R)$ as the pullback
\[\begin{tikzcd}
	\L(\R) \ar[r, "\tau^{-1}"] \ar[d] &  \KO\mathopen{}\left[\tfrac 1 2\right]\mathclose{}\ar[d,"{\mathrm{ch}}"] \\
	\H\Z_{(2)}\mathopen{}\left[t^{\pm 1}\right]\mathclose{} \ar[r,"{t \mapsto \frac{t}{4}}"] &\H\mathbb Q\mathopen{}\left[t^{\pm 1}\right]\mathclose{};
\end{tikzcd}\]
here the standard generator $x \in \L_4(\R) \cong \mathbb Z$, represented by the chain complex $\mathbb R[-2]$ equipped with the symmetric form $1$, is taken to $t$ along the left vertical map (determining it up to homotopy) and the right vertical map is obtained by composing the usual Chern character $\mathrm{ch} \colon \KU \rightarrow \H\mathbb Q\mathopen{}\left[u^{\pm 1}\right]$ with $|u| = 2$, that takes the Bott element $\beta \in \KU_2$ to $u$, with the complexification map $\mathrm c \colon \KO \rightarrow \KU$, and then lifting the result through the ring map $\H\mathbb Q\mathopen{}\left[t^{\pm 1}\right] \rightarrow \H\mathbb Q\mathopen{}\left[u^{\pm 1}\right], t \mapsto u^2$. It is then a theorem of \cite{KandL}, that $x$ is taken to $\beta^2/4$ by the composite 
\[\L(\mathbb R)\mathopen{}\left[\tfrac 1 2\right]\mathclose{} \stackrel{\tau^{-1}}{\longrightarrow} \KO\mathopen{}\left[\tfrac 1 2\right]\mathclose{} \stackrel{c}{\longrightarrow} \KU\mathopen{}\left[\tfrac 1 2\right]\mathclose{}\]
resulting in the square above.
\item One might hope that a splitting of $\L^{\q}(\Z)$ as in the theorem might hold even as $\L^{\s}(\Z)$-modules. This is, however, ruled out by the Anderson duality statement since $\L^{\s}(\Z)$ is easily checked to be indecomposable as an $\L^{\s}(\Z)$-module.
\end{enumerate}

Finally, let us mention that the Anderson duality between quadratic and symmetric L-theory as described in \cref{thmA} was anticipated by work on Sullivan's characteristic variety theorem initiated by Morgan and Pardon \cite{Pardon}: Pardon investigated the intersection homology Poincar\'e bordism spectrum $\Omega^{\mathrm{IP}}$, and showed that its homotopy groups are isomorphic to those of $\L^{\s}(\Z)$ in all non-negative degrees except $1$. By an ad hoc procedure reminiscent of Anderson duality compounded with periodisation, he furthermore produced a $4$-periodic cohomology theory from $\Omega^{\mathrm{IP}}$ equivalent to that obtained from the spectrum $\L^{\q}(\Z)$. 

The expectation that the isomorphisms $\Omega_n^{\mathrm{IP}} \cong \L_n^s(\Z)$ are induced by a map of spectra $\Omega^{\mathrm{IP}} \to \L^{\s}(\Z)$, was finally implemented by Banagl, Laures and McClure \cite{BLM}, informally by sending an IP-space to its intersection cochains of middle perversity.

This article partially arose from an attempt to understand Pardon's constructions directly from an L-theoretic perspective.

\subsection{The genuine variants}
We also investigate the genuine L-spectra $\L^{\gs}(\Z)$ and $\L^{\gq}(\Z)$ introduced in \cite{CDHII}. These variants of $\L$-theory are designed to fit into fibre sequences
\[\mathrm K(R)_{\mathrm{hC}_2} \lto \mathrm{GW}^{s}(R) \lto \L^{\gs}(R) \quad \text{and} \quad \mathrm K(R)_{\mathrm{hC}_2} \lto \mathrm{GW}^{q}(R) \lto \L^{\gq}(R),\]
where the middle terms denote the symmetric and quadratic Grothendieck-Witt spectra, variously also denoted by $\mathrm{KO}(R)$ and $\mathrm{KO}^q(R)$, respectively. In particular, the groups $\L_0^\gs(R)$ and $\L_0^\gq(R)$ are exactly the classical Witt groups of unimodular symmetric and quadratic forms over $R$, respectively; recall that Witt groups are given by isomorphism classes of such forms, divided by those that admit Lagrangian subspaces. 
 
In fact, one of the main results of \cite{CDHIII} is an identification of the homotopy groups of the spectrum $\L^\gs(R)$ with Ranicki's initial definition of symmetric L-groups in \cite{Ranicki4}, which lack the $4$-periodicity exhibited by $\L^{\q}(R)$ and $\L^{\s}(R)$. The spectra $\L^\gs(\Z)$ and $\L^\gq(\Z)$ are thus hybrids of the classical quadratic and symmetric $\L$-spectra: 
For any commutative ring, there are canonical maps
\[\L^{\q}(R) \longrightarrow \L^\gq(R) \stackrel{\mathrm{sym}}{\lto} \L^\gs(R) \longrightarrow \L^{\s}(R),\]
of which the middle one forgets the quadratic refinement and the left hand map induces an isomorphism on homotopy groups in degrees below $2$. For Dedekind rings, like the integers, the right hand map is an isomorphism in non-negative degrees and the middle map in degrees outside $[-2,3]$. 

We also recall from \cite{CDHIV} that $\L^\gs(R)$ is an $\E_\infty$-ring spectrum, the right hand map refines to an $\E_\infty$-ring maps and the entire displayed sequence consists of $\L^\gs(R)$-module spectra. By the calculations described above, the periodicity generator $x \in \L_4^\s(\Z)$ pulls back from $\L^\gq_4(\Z)$ and the $\L^\gs(\Z)$-module structure of $\L^\gq(\Z)$ then determines a map $\L^\gs(\Z)[4] \to \L^\gq(\Z)$,
which is an equivalence. 
As the second result of this paper, we then have the following analogue of \cref{thmA}, where we denote by $\ell(\R)$ the connective cover of $\L(\R)$:

\begin{introthm}\label{thmB}
As $\L^{\gs}(\Z)$-module spectra, $\L^{\gq}(\Z)$ and $\L^{\gs}(\Z)$ are Anderson dual to each other. Furthermore, there is a canonical $\E_1$-map $\ell(\R) \to \L^{\gs}(\Z)$ and an equivalence
\[ \L^{\gs}(\Z) \simeq \LLL \oplus (\ell(\R)/2)[1] \oplus (\L(\R)/(\ell(\R),2))[-2]\]
of $\ell(\R)$-modules, where $\LLL$ is given by the pullback
\[\begin{tikzcd}
	\LLL \ar[r] \ar[d] & \L(\R) \ar[d] \\
	\ell(\R)/8 \ar[r] & \L(\R)/8
\end{tikzcd}\]
of $\ell(\R)$-modules. Finally, as plain spectra, we have
\[\ell(\R)/2 \simeq \bigoplus_{k \geq 0} \H\Z/2[4k] \quad \text{and} \quad \L(\R)/(\ell(\R),2) \simeq \bigoplus_{k < 0} \H\Z/2[4k].\]
\end{introthm}

Again some comments are in order.
\begin{enumerate}
\item The homotopy groups of $\LLL$ are isomorphic to those of $\L(\R)$, but the induced map $\pi_*(\LLL) \rightarrow \L_*(\R)$ is an isomorphism only in non-negative degrees, whereas it is multiplication by $8$ below degree $0$.
\item By the discussion above, the canonical $\E_\infty$-map $\L^\gs(\Z) \to \L^{\s}(\Z)$ induces an equivalence on connective covers and the map $\ell(\R) \to \L^\gs(\Z)$ is then obtained from the ring map $\L(\R) \to \L^{\s}(\Z)$ of \cref{thmA} by passing to connective covers. 
\item Conversely, denoting by $x$ a generator of $\L^\gs_4(\Z) = \L^{\s}_4(\Z)$ there are canonical equivalences
\[\L^\gs(\Z)[x^{-1}] \simeq \L^{\s}(\Z) \quad \text{and} \quad \ell(\R)[x^{-1}] \simeq \LLL[x^{-1}] \simeq \L(\R),\]
which translate the splitting of \cref{thmB} into that of \cref{thmA}. 
\item Similarly, the Anderson duality statement of \cref{thmB} has that of \cref{thmA} as an immediate consequence, since $\L^{\q}(\Z) = \mathrm{div_x}\L^\gq(\Z)$ consists of the $x$-divisible part of $\L^\gq(\Z)$, and Anderson duality generally takes the inversion of an element $x$ to the $x$-divisible part of the dual, as it sends colimits to limits. We will, however, deduce the Anderson duality in \cref{thmB} from that of the classical $\L$-spectra in \cref{thmA}.

\end{enumerate}

\subsection{A possible multiplicative splitting of $\L^{\s}(\Z)$}\label{question}

While our results explain the precise relation of the spectra $\L^{\s}(\Z)$ and $\L(\R)$, they give no information about the multiplicative structures carried by them. Through various lingering questions we are lead to wonder:

\begin{conj}
Is the ring map $\L^{\s}(\Z) \rightarrow \L(\R)$, induced by the homomorphism $\Z \rightarrow \R$, a split square-zero extension?
\end{conj}

\cref{thmA}, in particular, supplies the requisite split in the realm of $\E_1$-spectra, though our question can equally well be considered at the level of $\E_k$-rings for any $k \in [1,\infty]$. For instance, we do not know whether our splitting can be promoted to an $\E_\infty$- or even an $\E_2$-map.

\subsection*{Organisation of the paper}
In \cref{sectiona,sectionl} we recall relevant facts about Anderson duality and $\L$-spectra, respectively. 
In \cref{sections,sectionq,sectiongs} we then analyse the spectra $\L^{\s}(\Z)$, $\L^{\q}(\Z)$ and $\L^\gs(\Z)$ in turn, proving \cref{thmA} in \cref{lstype}, \cref{adlslq}, \cref{symmetric case} and \cref{thmB} in \cref{AD for genuine L}, \cref{abc}. 

The appendix contains the computation of the homotopy ring of the cofibre $\L^{\n}(\Z)$ of the symmetrisation map $\L^{\q}(\Z) \rightarrow \L^{\s}(\Z)$, which we use at several points. The result is well-known, but we had difficulty locating a proof in the literature.

\subsection*{Notation and Conventions}
Firstly, as already visible in the introduction, we will adhere to the naming scheme for $\L$-theory spectra introduced by Lurie in \cite{LurieL}, e.g.\ writing $\L^{\s}(R)$ and $\L^{\q}(R)$ rather than Ranicki's $\L^\bullet(R)$ and $\L_\bullet(R)$. We hope that this will cause no confusion, but explicitly warn the reader of the clash that $\L^{\n}(R)$ for us will mean the normal $\L$-spectrum of a ring $R$, variously called $\widehat\L_\bullet(R)$ and $\mathrm{NL}(R)$ by Ranicki, and never the $n$-th symmetric $\L$-group of $R$. 

Secondly, essentially all of our constructions will only yield connected (as opposed to contractible) spaces of choices. Contrary to standard use in homotopical algebra, we will therefore use the term \emph{unique} to mean unique up to homotopy throughout in order to avoid awkward language. Similarly, we will say that the choices form a $G$-torsor if the components of the space of choices are acted upon transitively and freely by the group $G$; we wish it understood that no discreteness assertion is contained in the statement.

Finally, while most of the ring spectra we will encounter carry $\E_\infty$-structures, we will need to consider $\E_1$-ring maps among them and therefore have to distinguish left and right modules. We make the convention that all module spectra occurring in the text will be considered as left module spectra.

\subsection*{Acknowledgements}
We wish to thank Felix Janssen, Achim Krause, Lennart Meier and Carmen Rovi for helpful discussions, and Baptiste Calm\`es, Emanuele Dotto, Yonatan Harpaz, Kristian Moi, Denis Nardin and Wolfgang Steimle for the exciting collaboration that lead to the discovery of the genuine $\L$-spectra. We also thank Mirko Stappert for comments on a draft, and Julie Bannwart for catching an oversight in a previous version.
Finally, we wish to heartily thank Michael Weiss and the (sadly now late) Andrew Ranicki for a thoroughly enjoyable eMail correspondence whose content makes up the appendix. \\

FH is a member of the Hausdorff Center for Mathematics at the University of Bonn funded by the German Research Foundation (DFG) under GZ 2047/1 project ID 390685813. ML was supported by the CRC/SFB 1085 `Higher Invariants' at University of Regensburg, by the research fellowship DFG 424239956, and by the Danish National Research Foundation through the Center for Symmetry and Deformation (DNRF92) and the Copenhagen Centre for Geometry and Topology. TN was funded by the DFG through grant no.\ EXC 2044 390685587, `Mathematics M\"unster: Dynamics--Geometry--Structure'.

\section{Recollections on Anderson duality}\label{sectiona}

We recall the definition of the Anderson dual of a spectrum and collect some basic properties needed for our main theorem.

\begin{Def}
Let $M$ be an injective abelian group. Consider the functor
\[\begin{tikzcd}[row sep=tiny]
	(\Sp)^{\mathrm{op}} \ar[r] & \Ab_\Z \\ 
	X \ar[r, mapsto] & \Hom(\pi_{-*}(X),M)
\end{tikzcd}\]
Since $M$ is injective, this is a cohomology theory, and thus represented by a spectrum $\I_M(\S)$.
Define $\I_M(X) = \map(X,\I_M(\S))$.
\end{Def}

For instance, the spectrum $\I_{\Q/\Z}(X)$ is known as the Brown--Comenetz dual of $X$.
Clearly, one obtains an isomorphism $\pi_*(\I_M(X)) \cong \Hom(\pi_{-*}(X),M)$, and homomorphisms $M \to M'$ of injective abelian groups induce maps $\I_M(X) \to \I_{M'}(X)$.

\begin{Def}
Define the Anderson dual $\I(X)$ of a spectrum $X$ to fit into the fibre sequence
\[ \I(X) \lto \I_\Q(X) \lto \I_{\Q/\Z}(X) .\]
\end{Def}

One immediately obtains:

\begin{Lemma}\label{anderson dual as mapping spectrum}
For any two spectra $X,Y$ there is a canonical equivalence
\[\map(Y,\I(X)) \simeq \I(X\otimes Y).\]
In particular, $\I(X) \simeq \map(X,\I(\S))$, so that $\I(X)$ canonically acquires the structure of an $R^\mathrm{op}$-module spectrum, if $X$ is an $R$-module spectrum for some ring spectrum $R$.
\end{Lemma}
\begin{proof}
The second statement is immediate from the definitions since $\map(X,-)$ preserves fibre sequences. The first statement then follows by adjunction:
\[ \map(Y,\I(X)) \simeq \map\big(Y,\map(X,\I(\S))\big) \simeq \map(X\otimes Y,\I(\S)) \simeq \I(X\otimes Y).\]
\end{proof}

One can calculate the homotopy groups of the Anderson dual by means of the following exact sequence.
\begin{Lemma}\label{homotopy groups of anderson dual}
Let $X \in \Sp$ be a spectrum. Then the fibre sequence defining $\I(X)$ induces exact sequences
\[ 0 \lto \Ext\big(\pi_{-n-1}(X),\Z\big) \lto \pi_{n}(\I(X)) \lto \Hom\big(\pi_{-n}(X),\Z\big) \lto 0\]
of abelian groups that split non-canonically. If $X$ is an $R$-module spectrum then this sequence is compatible with the $\pi_*(R)^\mathrm{op}$-module structure on the three terms.
\end{Lemma}
\begin{proof}
From the fibre sequence
\[ \I(M) \lto \I_\Q(M) \lto \I_{\Q/\Z}(M) \]
we obtain a long exact sequence of homotopy groups. Since $\Q$ and $\Q/\Z$ are injective abelian groups this sequence looks as follows
\begin{multline*}
\Hom(\pi_{-n-1}(X),\Q) \longrightarrow \Hom(\pi_{-n-1}(X),\Q/\Z) \\ \longrightarrow \pi_{-n}(\I(M)) \longrightarrow \Hom(\pi_{-n}(X),\Q) \longrightarrow \Hom(\pi_{-n}(X),\Q/\Z) 
\end{multline*}
Since the sequence
\[ 0 \lto \Z \lto \Q \lto \Q/\Z \lto 0\]
is an injective resolution of $\Z$ it follows that the cokernel of the left most map is given by $\Ext(\pi_{-n-1}(X),\Z)$, and the kernel of the right most map is given by $\Hom(\pi_{-n}(X),\Z)$ as needed.

{To see that the sequence splits simply note that $\Ext(B,C)$ is always a cotorsion group, which by definition means that every extension of any torsionfree group $A$ (here $A = \Hom\big(\pi_{-n}(X),\Z\big)$) by $\Ext(B,C)$ splits; see for example \cite[Chapter 9, Theorem 6.5]{Fuchs}. For the reader's convenience we give the short proof: We generally have
\begin{align*}
\Ext(A, \Ext(B,C)) & =  \pi_{-2}\map_{\H\Z}(\H A, \map_{\H\Z}(\H B, \H C)) \\
& = \pi_{-2}\map_{\H\Z}(\H A \otimes_{\H\Z} \H B, \H C) \\
& =  \Ext\left( \mathrm{Tor}(A, B), C)\right).
\end{align*}
Now if $A$ is torsionfree it is flat, so that the $\mathrm{Tor}$-term vanishes.}
\end{proof}

\begin{Example}
From this, one immediately obtains canonical equivalences $\I(\H F) \simeq \H(\mathrm{Hom}(F,\mathbb Z))$ for every free abelian group $F$ and for every torsion abelian group $T$ we find that $\I(\H T) \simeq \H(\mathrm{Hom}(T,\mathbb Q/\mathbb Z))[-1]$. In particular, $\H\Z$ is Anderson self-dual and $\I(\H\Z/n) \simeq \H\Z/n[-1]$.
\end{Example}

Applying \cref{homotopy groups of anderson dual} to $X \otimes Y$ we find: 
\begin{Lemma}\label{coefficient sequence for anderson dual}
Let $X,Y \in \Sp$ be spectra. Then there is a canonical short exact sequence
\[ 0 \lto \Ext(X_{n-1}(Y),\Z) \lto \I(X)^n(Y) \lto \Hom_\Z(X_n(Y),\Z) \lto 0\]
\end{Lemma}

\begin{proof}
Note only that $X_{n}(Y) = \pi_n(X \otimes Y)$ by definition and 
\[\I(X)^n(Y) \cong \pi_{-n}\big(\map(Y,\I(X))\big) \cong \pi_{-n}\I(X \otimes Y)\]
by \cref{anderson dual as mapping spectrum}.
\end{proof}

Applied to $X=\H\Z$ this sequence reproduces the usual universal coefficient sequence for integral cohomology, but it has little content for example for $X=\H\Z/n$. 
The following standard lemma will be useful to identify Anderson duals:

\begin{Lemma}\label{free homotopy implies free module}
Let $R$ be a ring spectrum and let $M$ be a module spectrum over $R$. If $\pi_*(M)$ is a free module over $\pi_*(R)$, then $M$ is a sum of shifts of the free rank 1 module $R$.
\end{Lemma}
\begin{proof}
Choose an equivalence $\pi_*(M) \cong \bigoplus\pi_*(R)[d_i]$. The base elements on the right correspond to maps $\S^{d_i} \to M$, which by linearity extend to a map \[\bigoplus R[d_i] \lto M.\]
One readily checks that this induces an isomorphisms on homotopy groups.
\end{proof}

\begin{Example}
\begin{enumerate}
\item Using \cref{free homotopy implies free module}, one immediately concludes that $\KU$ is Anderson self-dual, which was Anderson's original observation in \cite{Anderson}. 
\item It follows that $\I(\KO) \simeq \KO[4]$: \cref{homotopy groups of anderson dual} implies that the homotopy groups of $\I(\KO)$ agree with those of $\KO[4]$. To see that the $\KO_*$-module structure is free of rank one on the generator in degree $4$, observe that the sequence of \cref{homotopy groups of anderson dual} is one of $\KO_*$-modules. This determines the entire module structure, except the $\eta$-multiplication $\pi_{8k+4}\I(\KO) \rightarrow \pi_{8k+5}\I(\KO)$, where the groups are $\Z$ and $\Z/2$, respectively. To see that this $\eta$-multiplication is surjective as desired, one can employ the following argument, that we will use again in the proof of \cref{thmA}. From the long exact sequence associated to the multiplication by $\eta$, the surjectivity of the map in question is implied by $\pi_{8k+5}(\I(\KO)/\eta) = 0$. Applying Anderson duality to the fibre sequence defining $\KO/\eta$ we find an equivalence $\I(\KO)/\eta\simeq\I(\KO/\eta)[2]$. But by Wood's equivalence $\KO/\eta \simeq \KU$, and the Anderson self-duality of $\KU$, these spectra are even, which gives the claim; for a proof of Wood's theorem see \cite[Theorem 3.2]{AkhilWood}, but note that this was cut from the published version \cite{AkhilnoWood}. 

A slightly different argument, based on the computation of the $\mathrm C_2$-equivariant Anderson dual of $\KU$ was recently given by Heard and Stojanoska in \cite{Drews}. 
\item Stojanoska also announced a proof that $\I(\mathrm{Tmf}) \simeq \mathrm{Tmf}[21]$ in \cite{Stoja}, having established the corresponding statement after inverting $2$ in \cite{Stojanoska}.
\end{enumerate}
\end{Example}

Finally, let us record:

\begin{Thm}\label{andersondualdual}
For every spectrum $X$ all of whose homotopy groups finitely generated, the natural map
\[X \longrightarrow \I^2(X)\]
adjoint to the evaluation
$X \otimes \map(X,\I(\S)) \longrightarrow \I(\S)$
is an equivalence. 
\end{Thm}

The proof is somewhat lengthy and can be found for example in \cite[Theorem 4.2.7]{DagXIV}. 
Note that this statement really fails for general $X$, for example in the cases $\H\Q$ or $\H F$, for $F$ a free abelian group of countable rank.

\section{Recollections on L-theory}\label{sectionl}

As many of the details will be largely irrelevant for the present paper, let us only mention that the various flavours of $\L$-groups are defined as cobordism groups of corresponding flavours of Poincar\'e chain complexes; see for example \cite{Ranicki} or \cite{LurieL} for details. Let us, however, mention, that in agreeance with \cite{CDHI} and \cite{LurieL} we will describe such Poincar\'e chain complexes by a hermitian form rather than the hermitian tensor preferred by Ranicki. Consequently, the $i$-th $\L$-group will admit those complexes as cycles whose Poincar\'e duality is $-i$-dimensional, i.e.\ makes degrees $k$ and $-k-i$ correspond. 

$\L$-theory spectra are then built by realising certain simplicial objects of Poincar\'e ads, the most highly structured result currently available being \cite[Theorem 18.1]{LauresMCCII}, where Laures and McClure produce a lax symmetric monoidal functor
\[\L^{\s} \colon \Ringinv \longrightarrow \Sp\]
assigning to a ring with involution its (projective) symmetric $\L$-theory; both source and target are regarded as symmetric monoidal via the respective tensor products. In particular, for a commutative ring $R$, the spectrum $\L^{\s}(R)$ is an $\E_\infty$-ring spectrum; here and throughout we suppress the involution from notation if it is given by the identity of $R$. Together with the fact that the (projective) quadratic and normal $\L$-spectra reside in a fibre sequence
\[\L^{\q}(R) \xrightarrow{\mathrm{sym}} \L^{\s}(R) \longrightarrow \L^{\n}(R)\]
with $\L^{\n}(R)$ an $\L^{\s}(R)$-algebra and $\L^{\q}(R)$ an $\L^{\s}(R)$-module, this will suffice for our results concerning the spectra $\L^{\s}(\Z)$ and $\L^{\q}(\Z)$ in \cref{sections,sectionq}; note that, in the current literature it is only established that $\L^{\n}(R)$ is an algebra up to homotopy and $\L^{\q}(R)$ a module up to homotopy and this will be enough for the current paper. Highly structured refinements are contained in \cite{CDHI, CDHIV}.

Let us also mention, that by virtue of the displayed fibre sequence, elements in $\L^{\n}_{k+1}(R)$ can be represented by a $k$-dimensional quadratic Poincar\'e chain complex equipped with a symmetric null-cobordism.

The genuine $\L$-theory spectra, however, are not covered by this regime (for example, they do not admit $\L^{\s}(\Z)$-module structures in general), and we recall the necessary additions from \cite{CDHI,CDHII,CDHIII} at the beginning of \cref{sectiongs}.\\

Largely to fix notation, let us recall the classical computation of the quadratic and symmetric $\L$-groups of the integers, full proofs appear for example in \cite[Section 4.3.1]{phonebook}, see also \cite[Lectures 15 \& 16]{LurieL}.

\begin{Thm}[Kervaire-Milnor, Mishchenko]\label{L(Z)}
The homotopy ring of $\L^{\s}(\Z)$ and the $\L_*^\s(\Z)$-module $\L^{\q}_*(\Z)$ are given by
\[\L^{\s}_*(\Z) = \Z[x^{\pm 1},e]/(2e,e^2) \quad \text{and} \quad \L_*^{\q}(\Z) = \L^{\s}_*(\Z)/(e) \oplus \big(\L^{\s}_*(\Z)/(2,e)\big)[-2]\]
with $|x|= 4$ and $|e|=1$.
\end{Thm}

Here, $x \in \L^{\s}_4(\Z)$ is represented by $\Z[-2]$ equipped with the standard symmetric form of signature $1$, $e \in \L^{\s}_1(\Z)$ is the class of $\Z/2[-1]$ with its unique non-degenerate symmetric form of degree $1$. The element $(1,0)$ on the right is represented by the $E_8$-lattice and $(0,1)$ is the class of $\Z^2[1]$ with its standard skew-symmetric hyperbolic form,  equipped with quadratic refinement $(a,b) \mapsto a^2+b^2+ab$. The symmetrisation map sends $(1,0)$ to $8x$ and $(0,1)$ to $0$.

In fact, $\L^{\q}_*(\Z)$ obtains the structure of a non-unital ring through the symmetrisation map, which is most easily described as
\[\L^{\q}_*(\Z) = 8\Z[t^{\pm1},g]/(16g,64g^2)\]
with $8t \in \L_4^\q(\Z)$ and $8g \in \L_{-2}(\Z)$ corresponding to $(x,0)$ and $(0,1)$, respectively.

\begin{Rmk}
Note  that the element $e$ is not the image of $\eta$ under the unit $\S \to \L^{\s}(\Z)$. In fact, the unit map $\S_* \rightarrow \L_*^\s(\Z)$ is trivial outside degree $0$, as follows for example by combining \cref{sigsquare} and \cref{thommaps} below.
\end{Rmk}

As a consequence of the above, one can immediately deduce the additive structure of the normal $\L$-groups of the integers. Including the ring structure, the result reads:

\begin{Prop}\label{ringln}
The homotopy ring of $\L^{\n}(\Z)$ is given by 
\[\L^{\n}_*(\Z) = \Z/8[x^{\pm 1},e,f]/\Big( 2e=2f=0, e^2=f^2 =0, ef= 4 \Big).\]
\end{Prop}

Here, $e$ and $x$ are the images of the corresponding classes in $\L_*^\s(\Z)$ and $f \in \L_{-1}^\n(\Z)$ is represented by the pair consisting of the $\Z^2[1]$ equipped by with the quadratic form representing $g$ together with the symmetric Lagrangian given by the diagonal. Under this identification, the boundary map $\L_{4i-1}^n(\Z) \rightarrow \L_{4i-2}^q(\Z)$ sends $x^if$ to $8x^ig$ and necessarily vanishes in all other degrees.

The only part of the statement not formally implied by \cref{L(Z)} is the equality $ef = 4$. As we had a hard time extracting a proof from the literature, but will crucially use this fact later, we supply a proof in the appendix.\\

Now, the $2$-local homotopy type of the fibre sequence
\[\L^{\q}(\Z) \xrightarrow{\mathrm{sym}} \L^{\s}(\Z) \longrightarrow \L^{\n}(\Z)\]
is well-known, the results first appearing in \cite{TW}. As we will have to use them, we briefly give a modern account of the relevant parts. The starting point of the analysis is Ranicki's signature square; see for example \cite[p. 290]{RanickiTSO}.

\begin{Thm}[Ranicki]\label{sigsquare}
The symmetric and normal signatures give a commutative square of homotopy ring spectra
\[\begin{tikzcd}
	\MSO \ar[r,"{\sigma^s}"] \ar[d,"{\mathrm{MJ}}"'] & \L^{\s}(\Z) \ar[d] \\
	\MSG \ar[r,"{\sigma^n}"] & \L^{\n}(\Z),
\end{tikzcd}\]
where $\MSG$ is the Thom spectrum of the universal stable spherical fibration of rank $0$ over $\BSG  = \BSl_1(\S)$ and $\MSO$ that of its pullback along $\mathrm J \colon \BSO \rightarrow \BSG$.
\end{Thm}

\begin{Rmk}
As part of their work in \cite{LauresMCCI, LauresMCCII} the top horizontal map was refined to an $\E_\infty$-map by Laures and McClure. Refining the whole square to one of $\E_\infty$-ring spectra is work in progress
%and forthcoming work 
by the first two authors with Laures;
%will refine the entire square to one of $\E_\infty$-rings \cite{HLLI}. 
we do not, however, have to make use of these extensions here.
\end{Rmk}

As a second ingredient we make use of the following result, the $2$-primary case of which originally appeared in \cite{HZThom}. A full proof appears for example in \cite[Theorem 5.7]{OmarTobi}.

\begin{Thm}[Hopkins-Mahowald]
For $p$ a prime, the Thom spectrum of the $\E_2$-map
\[\eta_p \colon \tau_{\geq 2}(\Omega^2 S^3) \longrightarrow \BSl_1(\S_{(p)})\]
induced by the generator $1-p \in \Z_{(p)}^\times = \pi_1\BGl_1(\S_{(p)})$ is given by $\H\Z_{(p)}$.
\end{Thm}
Here, $\BGl_1(\S_{(p)})$ is the classifying space for $p$-local stable spherical fibrations, and $\BSl_1(\S_{(p)})$ is its universal cover. Note also, that the $\E_2$-structure on $\H\Z$ resulting from the theorem is the restriction of its canonical $\E_\infty$-structure (since for any $k$ the spectrum $\H R$ admits a unique $\E_k$-ring structure refining the ring structure on $R$).

For $p=2$ the map occurring in the theorem evidently factors as 
\[\tau_{\geq 2}(\Omega^2 S^3) \stackrel{\eta}{\lto} \BSO \stackrel{J}{\lto} \BSG = \BSl_1(\S) \longrightarrow \BSl_1(\S_{(2)}),\]
but the Thom spectrum of this integral refinement is neither $\H\Z$ nor $\H\Z_{(2)}$ before localisation at $2$. Taking on the other hand the product over all primes we arrive at an $\E_2$-map
\[\eta_\mathbb P \colon \tau_{\geq 2}(\Omega^2 S^3) \longrightarrow \BSG,\]
since its target is the product over all the $\BSl_1(\S_{(p)})$ by the finiteness of $\pi_*\S$ in positive degrees. The Thom spectrum of $\eta_\mathbb P$ is then necessarily $\H\Z$. We obtain:

\begin{Cor}\label{thommaps}
Applying Thom spectra to $\eta_2$ and $\eta_\mathbb P$ results in a commutative square
\[\begin{tikzcd}
	\H\Z \ar[r,"{\mathrm M\eta_2}"] \ar[d,"{{\mathrm M\eta_\mathbb P}}"] & \MSO_{(2)} \ar[d] \\
         \MSG \ar[r] & \MSG_{(2)}
 \end{tikzcd}\]
of $\E_2$-ring spectra.
\end{Cor}

Putting these together we obtain:

\begin{Cor}\label{waystosplit}
Any $2$-local (homotopy) module spectrum $E$ over $\tau_{\geq 0}\L^{\s}(\Z)$ acquires a preferred (homotopy) $\H\Z$-module structure through the composite
\[\H\Z \xrightarrow{\mathrm M\eta_2} \MSO_{(2)} \xrightarrow{\sigma^s}\tau_{\geq 0}\L^{\s}(\Z)_{(2)}. \]
In particular, there is a $\prod_{i \in \Z}\mathrm{Ext}(\pi_i(E),\pi_{i+1}(E))$-torsor of $\H\Z$-linear equivalences
\[E \simeq \bigoplus_{i \in \Z} \H\pi_i(E)[i]\]
inducing the identity map on homotopy groups.
\end{Cor}

\begin{proof}
Recall that an $\H\Z$-module spectrum always admits a splitting as displayed, by first choosing maps $M\pi_i(E)[i] \rightarrow E$ from the Moore spectrum inducing the identity on $\pi_i$, and then forming 
\[\H\pi_i(E)[i] \simeq \H\Z \otimes M(\pi_i(E))[i] \longrightarrow E\]
using the $\H\Z$-module structure. The indeterminacy of such an equivalence is clearly the direct product over all $i$ of those components in $\pi_0\map(\H\pi_i(E)[i], E)$ inducing the identity in homotopy. Since 
\[\pi_0\map_{\H\Z}(\H\pi_i(E)[i], E) \cong \Hom(\pi_iE,\pi_iE) \oplus \mathrm{Ext}(\pi_i(E),\pi_{i+1}(E))\]
the claim follows.
\end{proof}

Applying \cref{waystosplit} to the $\L^{\s}(\Z)$-module $\L^{\n}(\Z)$, one obtains:
\begin{Cor}[Taylor-Williams]\label{splitln}
The equivalences
\[\L^{\n}(\Z) \simeq \bigoplus_{k \in \mathbb Z} \Big[\H\Z/8[4k] \oplus \H\Z/2[4k+1] \oplus \H\Z/2[4k+3]\Big]\]
of $\H\Z$-modules compatible with the action of the periodicity operator $x \in \L^{\s}_4(\Z)$ form a $(\Z/2)^2$-torsor.
\end{Cor}

\begin{proof}
One easily checks that $x$ acts on the indeterminacy
\[\prod_{k \in \Z}\mathrm{Ext}(\L^{\n}_{4k}(\Z),\L^{\n}_{4k+1}(\Z)) \oplus \mathrm{Ext}(\L^{\n}_{4k+3}(\Z),\L^{\n}_{4k+4}(\Z)) \cong \prod_{k \in \Z} (\Z/2)^2\]
by shifting, which immediately gives the result.
\end{proof}

\begin{Rmk}\label{E3}
\begin{enumerate}
\item Hopkins observed that the map $\eta_2$, and therefore also $M\eta_2$, are in fact $\E_3$-maps via the composite
\[\B S^3 \simeq \B \mathrm{Sp}(1) \longrightarrow \B\mathrm{Sp} \xrightarrow\beta \B^5\mathrm O \xrightarrow{\eta} \B^4 \mathrm O,\]
where $\beta$ denotes the Bott map. 
\item The Thom spectrum of the map $\Omega^2 S^3 \longrightarrow \BSl_1(\S_{(p)})$ is $\H\F_p$ by another computation of Mahowald in \cite{HZThom}. In particular, there is also an $\E_3$-map $\mathrm M \eta_2 \colon \H\F_2 \to \mathrm{MO}$
fitting into the evident diagrams (involving $\mathrm M\G$) with the maps from \cref{thommaps}.
\item Note that not even an $\E_1$-map $\H\Z \rightarrow \L^{\s}(\Z)$ can exist before localisation at $2$: By \cite[Theorem 5.2]{KandL} there is a canonical equivalence $\L(\R)\mathopen{}\left[\tfrac 1 2\right]\mathclose{} \simeq \KO\mathopen{}\left[\tfrac 1 2\right]\mathclose{}$, and $\Omega^{\infty}\KO\mathopen{}\left[\tfrac 1 2\right]\mathclose{}$ is not a generalised Eilenberg-Mac Lane space, for example on account of its $\mathbb F_p$-homology.
\item Using the ring structure on $\L^{\n}_*(\Z)$, Taylor and Williams \cite{TW} in fact also determine the homotopy class of the multiplication map of $\L^{\n}(\Z)$ under the splitting in \cref{splitln}, but we shall not need that result.
\end{enumerate}
\end{Rmk}

\section{The homotopy type of $\L^{\s}(\mathbb Z)$}\label{sections}

To give the strongest forms of our results concerning the $\L$-spectra of the integers, we need to analyse their relation to the ring spectrum $\L(\R)$. From \cite[Proposition 22.16 (i)]{Ranicki} we first find, see also \cite[Lecture 13]{LurieL}:

\begin{Prop}\label{pilr}
The homotopy ring of $\L(\R)$ is given by 
\[\L_*(\R) = \Z[x^{\pm1}], \]
where $x \in \L_4(\R)$ denotes the image of the class in $\L_4^\s(\Z)$ of the same name. 
\end{Prop}

Recall from \cref{thommaps} that $\L(\R)_{(2)}$ receives a preferred $\E_2$-map from $\H\Z$, and is therefore, in particular, an $\E_1$-$\H\Z$-algebra.

\begin{Cor}\label{free guy}\label{section of L-theories}
The localisation $\L(\R)_{(2)}$ is a free $\E_1$-$\H\Z_{(2)}$-algebra on an invertible generator of degree $4$, namely $x \in \L_4(\R)_{(2)}$.
\end{Cor}

Recall that the free $\E_1$-$\H\Z_{(2)}$-algebra on an invertible generator $t$ is given as the localisation of
\[\H\Z[t] = \bigoplus_{n \in \mathbb N} \H\Z[4]^{\otimes_{\H\Z} n}\] 
at $2$ and $t$. In particular, its homotopy ring is given by the ring of Laurent polynomials over $\Z_{(2)}$. 

\begin{proof}
By definition of the source there exists a canonical map $\H\Z[t] \to \L(\R)_{(2)}$ which sends $t$ to $x$. Since $x$ is invertible, this map factors through the localisation $\H\Z[t^{\pm 1}]$ and the resulting map is an equivalence by \cref{pilr}.
\end{proof}

\begin{Cor}\label{lzlrsplit}
There is a unique $\E_1$-section of the canonical map $\L^{\s}(\Z) \to \L(\R)$, that is a section of $\H\Z$-algebras after localisation at $2$.
\end{Cor}

In particular, every $\L^{\s}(\Z)$-module acquires a canonical $\L(\R)$-module structure.

\begin{proof}

We denote by $\Gamma(A \to B)$ the space of sections of a map $A \to B$ of $\E_1$-algebras. Then from the usual fracture square, or its formulation as a pullback of $\infty$-categories, one gets a cartesian square
\[\begin{tikzcd}
	\Gamma\left(\L^{\s}(\Z) \to \L(\R)\right) \ar[r]\ar[d] & \Gamma(\L^{\s}(\Z)_{(2)} \to \L(\R)_{(2)}) \ar[d] \\
	\Gamma\left(\L^{\s}(\Z)\!\left[\tfrac 1 2\right] \to \L(\R)\!\left[\tfrac 1 2\right]\right)  \ar[r] & \Gamma\left(\L^{\s}(\Z)_\Q \to 				\L(\R)_\Q\right) \ .
\end{tikzcd}\]
But the lower maps $\L^{\s}(\Z)\!\left[\tfrac 1 2\right] \to \L(\R)\!\left[\tfrac 1 2\right]$ and $\L^{\s}(\Z)_\Q \to \L(\R)_\Q$ are equivalences so that the associated spaces of sections are contractible. We conclude that 
the upper horizontal map is an equivalence. Thus a section is completely determined by its 2-local behaviour. If we require this 2-local section to be an $\H\Z$-algebra map then it follows from the freeness of $\L(\R)_{(2)}$ as an $\H\Z$-algebra (see \cref{lzlrsplit} above) that the section is unique. In fact the space of $\H\Z$-algebra sections of $\L^{\s}(\Z)_{(2)} \to \L(\R)_{(2)}$ is equivalent to the fibre of 
\[\Omega^{\infty+4}\L^{\s}(\Z)_{(2)} \longrightarrow \Omega^{\infty+4}\L(\R)_{(2)},\]
over $x$ which has vanishing $\pi_0$ by \cref{pilr}.
\end{proof}

For later use we also record:

\begin{Prop}\label{self dual}
The spectrum $\L(\R)$ is canonically Anderson self-dual as a module over itself. In fact, the same is true for $\L(\bC,c)$, where $c$ denotes complex conjugation on $\bC$, and $\I(\L(\bC)) \simeq \L(\bC)[-1]$.
\end{Prop}
\begin{proof}
The homotopy rings in the latter two cases are $\L_*(\bC) \cong \mathbb F_2[x^{\pm1}]$ for $x$ the image of the periodicity element in $\L^{\s}_4(\mathbb Z)$ and $\L_*(\bC,c) \cong \mathbb Z[s^{\pm 1}]$ where $s^2 = x$; see \cite[Proposition 22.16]{Ranicki}. Using \cref{homotopy groups of anderson dual} one calculates that $\pi_*\I(R)$ is a free $\pi_*R$-module of rank 1 in each case. \cref{free homotopy implies free module} then finishes the proof.
\end{proof}

\begin{Rmk}
Recall from \cite{KandL} that the spectra $\L(\bC,c)$ and $\KU$ are closely related, but only equivalent after inverting $2$, despite having isomorphic homotopy rings.
\end{Rmk}

We now start with the proof of \cref{thmA}. As preparation we set:

\begin{Def}
Denote by $\dR$ the $\L^{\s}(\Z)$-module spectrum fitting into the fibre sequence
\[\dR \longrightarrow \L^{\s}(\Z) \longrightarrow \L(\R).\]
\end{Def}

We shall refer to $\dR$ as the de Rham-spectrum, since its homotopy groups are detected by de Rham's invariant $\L_{4k+1}^s(\Z) \rightarrow \Z/2$.

\begin{Cor}\label{drtype}
The spectrum $\dR$ is a $2$-local $\L^{\s}(\Z)$-module, whose underlying $\L(\R)$-module admits a unique equivalence to $(\L(\R)/2)[1]$ and whose underlying $\H\Z$-module admits a splitting
\[\dR \simeq \bigoplus_{k \in \Z} \H\Z/2[4k+1]\]
unique up to homotopy.
\end{Cor}

Note, that we do not claim here, that $\dR \simeq (\L(\R)/2)[1]$ as $\L^{\s}(\Z)$-modules, though this is clearly required for a positive answer to our question in \cref{question}.

\begin{proof}
It is immediate from the computation of the homotopy groups of $\L^{\s}(\Z)$ and $\L(\R)$ that multiplication with $e \in \L_1^s(\Z)$ gives an isomorphism
\[\big(\L_*^\s(\Z)/(2,e)\big)[1] \longrightarrow \dR\] 
of $\L^{\s}_*(\Z)$-modules. It then follows from the exact sequence induced by multiplication with $2$, that there is a unique homotopy class $\S^1/2 \rightarrow \dR$ representing $e$ and extending $\L(\R)$-linearly using \cref{lzlrsplit}, one obtains a canonical $\L(\R)$-module map
\[(\L(\R)/2)[1] \longrightarrow \dR.\]
This is readily checked to be an equivalence. The statement about the underlying $\H\Z$-module is immediate from \cref{waystosplit}.
\end{proof}

We obtain:

\begin{Cor}\label{lstype}
There is a unique equivalence 
\[\L^{\s}(\Z) \simeq \L(\R) \oplus (\L(\R)/2)[1]\]
of $\L(\R)$-modules, whose composition with the projection to the first summand agrees with the canonical map. In particular, there is a preferred equivalence
\[\L^{\s}(\Z) \simeq \L(\R) \oplus \bigoplus_{k \in \Z} \H\Z/2[4k+1]\]
of underlying spectra.
\end{Cor}

\begin{proof}
For the existence statement note simply that the fibre sequence 
\[\dR \longrightarrow \L^{\s}(\Z) \longrightarrow \L(\R)\]
is $\L(\R)$-linearly split, since its boundary map $\L(\R) \rightarrow \dR[1]$ vanishes as an $\L(\R)$-linear map because $\pi_{-1}\dR = 0$. Then the previous corollary gives the claim about the underlying spectrum. 

For the uniqueness assertion we compute
\begin{align*}
\pi_0\mathrm{end}_{\L(\R)}(\L(\R) \oplus (\L(\R)/2)[1]) =\ &\pi_0\mathrm{end}_{\L(\R)}(\L(\R)) \oplus \pi_0\map_{\L(\R)}(\L(\R), (\L(\R)/2)[1]) \ \oplus \ \\ &  \pi_0\mathrm{map}_{\L(\R)}((\L(\R)/2)[1], \L(\R)) \oplus \pi_0\mathrm{end}_{\L(\R)}((\L(\R)/2)[1]) \\
=\ & \Z \oplus 0 \oplus 0 \oplus \Z/2
\end{align*}
using the exact sequences associated to multiplication by $2$. The two non-trivial summands are detected by the effect on $\pi_0$ and $\pi_1$, respectively, and thus determined by compatibility with the map $\L^{\s}(\Z) \rightarrow \L(\R)$ and by being an equivalence.
\end{proof}

\begin{Rmk}\label{generalsplits}
In fact, \cite[Theorem 4.2.3]{Patchkoria} gives a triangulated equivalence \[\mathrm{ho}(\mathrm{Mod}(\L(\R)) \simeq \mathrm{ho}(\mathcal D(\Z[t^{\pm1}])),\] under which the homotopy groups of a module over $\L(\R)$ are given by the homology groups of the corresponding chain complex over $\mathcal D(\Z[t^{\pm1}])$. In particular, any $\L(\R)$-module $M$ splits (non-canonically in general) into
\[\bigoplus_{i=0}^3\L(\R) \otimes \mathrm M(\pi_i(M))[i],\]
where $\mathrm M(A)$ is the Moore spectrum of the abelian group $A$, though this statement is also easy to see by hand. 

Note, however, that the $\infty$-categories $\mathrm{Mod}(\L(\R))$ and $\mathcal D(\Z[t^{\pm1}])$ themselves are not equivalent, as the right comes equipped with an $\H\Z$-linear structure, whereas the left cannot admit one: $\L(\R)$ itself occurs as a morphism spectrum and does not admit an $\H\Z$-module structure as observed in the comments preceding \cref{lzlrsplit}.
\end{Rmk}

\section{The homotopy type of $\L^{\q}(\mathbb Z)$}\label{sectionq}

We start with the first part of \cref{thmA}:

\begin{Thm}\label{adlslq}
As modules over $\L^{\s}(\Z)$, there is a preferred homotopy class of equivalences
\[ \I(\L^{\s}(\Z)) \simeq \L^{\q}(\Z).\]
\end{Thm}

\begin{proof}
Since the homotopy groups of $\L^{\q}(\Z)$ and $\L^{\s}(\Z)$ are finitely generated, this is equivalent to showing that $\I(\L^{\q}(\Z)) \simeq \L^{\s}(\Z)$ as $\L^{\s}(\Z)$-modules by \cref{andersondualdual}. By \cref{free homotopy implies free module} it suffices to prove that $\pi_*(\I(\L^{\q}(\Z)))$ is a free module of rank 1 over $\pi_*(\L^{\s}(\Z))$ with a preferred generator. 

Now, one readily calculates that the homotopy groups of $\I(\L^{\q}(\Z))$ are isomorphic to those of $\L^{\s}(\Z)$ using \cref{homotopy groups of anderson dual} and $[E_8] \in \L_0^\q(\Z)$ determines a generator of $\pi_0(\I(\L^{\q}(\Z)))$.
It remains to determine the module action of $\pi_*(\L^{\s}(\Z))$ on $\pi_*\I(\L^{\q}(\Z))$, which is somewhat tricky (just as in the case of $\KO$), as the latter has contributions from both sides of the exact sequence of \cref{homotopy groups of anderson dual}. 

Being invertible, the action of $x$ is determined by the structure of the homotopy groups (up to an irrelevant sign), so the claim holds if and only if the map
\begin{equation}\label{times e}
	\pi_{4k}\I(\L^{\q}(\Z)) \stackrel{\cdot e}{\lto} \pi_{4k+1}\I(\L^{\q}(\Z)) \tag{$\ast$}
\end{equation}
is surjective. For this we consider the cofibre sequence
\[  \I(\L^{\q}(\Z))[1] \stackrel{\cdot e}{\lto} \I(\L^{\q}(\Z)) \lto \I(\L^{\q}(\Z))/e \]
and see that the surjectivity of \eqref{times e} is equivalent to the statement that
\[ \pi_{4k+1}\big(\I(\L^{\q}(\Z))/e\big) = 0.\]
Applying Anderson duality to the cofibre sequence
\[ \L^{\q}(\Z)[1] \stackrel{\cdot e}{\lto} \L^{\q}(\Z) \lto \L^{\q}(\Z)/e \]
one finds that
\[ \I(\L^{\q}(\Z))/e \simeq \I\big(\L^{\q}(\Z)/e\big)[2]\]
and thus that 
\[ \pi_{4k+1}\big(\I(\L^{\q}(\Z))/e\big) \cong \pi_{4k-1}\I(\L^{\q}(\Z)/e), \]
which can be computed by \cref{homotopy groups of anderson dual} as soon as we know the homotopy groups of $\L^{\q}(\Z)/e$. Spelling this out, we find that the map \eqref{times e} is surjective if and only if 
\begin{enumerate}
\item[(i)] $\pi_{4k+1}(\L^{\q}(\Z)/e)$ is torsion, and
\item[(ii)] $\pi_{4k}(\L^{\q}(\Z)/e)$ is torsionfree.
\end{enumerate}
We can calculate the homotopy groups of $\L^{\q}(\Z)/e$ by means of the following two cofibre sequences
\[\begin{tikzcd}[row sep=tiny]
	\L^{\q}(\Z)[1] \ar[r,"{\cdot e}"] & \L^{\q}(\Z) \ar[r] & \L^{\q}(\Z)/e \\
	\L^{\q}(\Z)/e \ar[r] & \L^{\s}(\Z)/e \ar[r] & \L^{\n}(\Z)/e.
\end{tikzcd}\]
Since in the homotopy groups of $\L^{\q}(\Z)$ the multiplication by $e$ is trivial (simply for degree reasons) we obtain short exact sequences
\[0 \lto \pi_*(\L^{\q}(\Z)) \lto \pi_*(\L^{\q}(\Z)/e) \lto \pi_{*-2}(\L^{\q}(\Z)) \lto 0, \]
which give
\[\pi_i(\L^{\q}(\Z)/e) \cong \begin{cases} M & \text{ if } i \equiv 0 (4) \\
							0 & \text{ if } i \equiv 1 (4) \\
							\Z\oplus \Z/2 & \text{ if } i \equiv 2 (4) \\
							0 & \text{ if } i \equiv 3 (4) \end{cases}\]
where $M$ is an extension of $\Z$ and $\Z/2$. In particular, we see that (i) in the above list holds true. To see (ii), we will calculate the extension $M$ by means of the second fibre sequence. A quick calculation in symmetric L-theory reveals that $\Z \cong \pi_{4i}(\L^{\s}(\Z)) \stackrel{\cong}{\to} \pi_{4i}(\L^{\s}(\Z)/e)$ and that $\pi_{4k+1}(\L^{\s}(\Z)/e) = 0$. We then consider the exact sequence
\[ 0 \lto \pi_{4i+1}(\L^{\n}(\Z)/e) \lto \pi_{4i}(\L^{\q}(\Z)/e) \lto \pi_{4i}(\L^{\s}(\Z)/e) \lto \pi_{4i}(\L^{\n}(\Z)/e) \lto 0\]
Hence $M$ is torsionfree if and only if the torsion group $\pi_{4i+1}((\L^{\n}(\Z))/e)$ is trivial.

To finally calculate this group we consider the exact sequence
\[\pi_{4k}(\L^{\n}(\Z)) \stackrel{\cdot e}{\lto} \pi_{4k+1}(\L^{\n}(\Z)) \lto \pi_{4k+1}(\L^{\n}(\Z)/e) \lto \pi_{4k-1}(\L^{\n}(\Z)) \stackrel{\cdot e}{\lto} \pi_{4k}(\L^{\n}(\Z)) \]
The arrow to the left is surjective since $xe \neq 0$ in $\pi_*(\L^{\n}(\Z))$. We thus find that 
\[ \pi_{4k+1}(\L^{\n}(\Z)/e) = \ker( \L^{\n}_{4k-1}(\Z) \stackrel{\cdot e}{\to} \L^{\n}_{4k}(\Z) ). \]
This kernel is trivial, as $ef \neq 0$ in $\L^{\n}_*(\Z)$ by \cref{ringln}.
\end{proof}

Before finishing the proof of \cref{thmA} by splitting $\L^{\q}(\Z)$, we collect a few results that allow us to describe the behaviour of the entire fibre sequence
\[\L^{\q}(\Z) \longrightarrow \L^{\s}(\Z) \longrightarrow \L^{\n}(\Z)\]
under Anderson duality.

\begin{Cor}\label{bla cor}
For any $\L^{\s}(\Z)$-module $X$, there is a canonical equivalence
\[ \map_{\L^{\s}(\Z)}(X,\L^{\q}(\Z)) \simeq \I(X)\]
as $\L^{\s}(\Z)$-modules.
\end{Cor}
\begin{proof}
We calculate
\begin{align*}
\map_{\L^{\s}(\Z)}(X,\L^{\q}(\Z)) & \simeq \map_{\L^{\s}(\Z)}(X,\map(\L^{\s}(\Z),\I(\S))) \\
		& \simeq \map_{\L^{\s}(\Z)}(\L^{\s}(\Z)\otimes X,\I(\S)) \\
		& \simeq \map(X,\I(\S)) \\
		& \simeq \I(X).
\end{align*}
using \cref{anderson dual as mapping spectrum} and \cref{adlslq}.
\end{proof}

\begin{Prop}\label{lnand}
There is a preferred equivalence $\I(\L^{\n}(\Z)) = \L^{\n}(\Z)[-1]$ as $\L^{\n}(\Z)$-module spectra.
\end{Prop}
\begin{proof}
We will again show that $\pi_*\I(\L^{\n}(\Z))$ is a free $\L^{\n}_*(\Z)$-module on a generator in degree $-1$ and apply \cref{free homotopy implies free module}.
The additive structure is immediate. Also, the action of $x \in \L^{\n}_4(\Z)$ is through isomorphisms as needed. We will show that $\pi_{4k}(\I(\L^{\n}(\Z)/e)) = 0$. This implies that multiplication by $e$ induces a surjection $\pi_{4k-1}(\I(\L^{\n}(\Z))) \to \pi_{4k}(\I(L^n(\Z)))$ and an injection $\pi_{4k-2}(\I(\L^{\n}(\Z))) \to \pi_{4k-1}(\I(\L^{\n}(\Z)))$. Using $ef=4$ this is easily checked to force the multiplication by all other elements to take the desired form.

As before, we have $\I(\L^{\n}(\Z))/e \simeq \I(\L^{\n}(\Z)/e)[2]$. \cref{homotopy groups of anderson dual}, together with the fact that all homotopy groups involved are torsion, gives
\[ \pi_{4k}(\I(\L^{\n}(\Z))/e) \cong \Ext(\pi_{-4k+1}(\L^{\n}(\Z)/e),\Z).\]
However, it follows from the description of the $e$-multiplication on $\L^{\n}(\Z)$ in \cref{ringln} that $\pi_{-4k+1}(\L^{\n}(\Z)/e) = 0$ as needed.
\end{proof}

\begin{Cor}\label{endos of Lq}
There are canonical equivalences
\[\mathrm{end}_{\L^{\s}(\Z)}(\L^{\q}(\Z)) \simeq \L^{\s}(\Z) \quad \text{and} \quad \map_{\L^{\s}(\Z)}(\L^{\n}(\Z),\L^{\q}(\Z)[1]) \simeq \L^{\n}(\Z),\]
the left hand one of $\E_1$-$\L^{\s}(\Z)$-algebras, the right hand one of $\L^{\s}(\Z)$-modules. Under the second equivalence the element $1 \in \pi_0\L^{\n}(\Z)$ corresponds to the boundary map of the symmetrisation fibre sequence.
\end{Cor}
Given the equivalence $\mathrm{end}_{\L^{\s}(\Z)}(\L^{\q}(\Z)) \simeq  \L^{\s}(\Z) $ in this Corollary, one may wonder whether $\L^{\q}(\Z)$ is an invertible $\L^{\s}(\Z)$-module. All that remains is to decide whether it is a compact $\L^{\s}(\Z)$-module, but alas, we were unable to do so; it is clearly related to our question in \cref{question}.

\begin{proof}
The displayed equivalences are direct consequences of \cref{bla cor}. To see that the left hand one preserves the algebra structures, simply note that the module structure of $\L^{\q}(\Z)$ provides an $\E_1$-map $\L^{\s}(\Z) \rightarrow \mathrm{end}_{\L^{\s}(\Z)}(\L^{\q}(\Z))$, whose underlying module map is necessarily inverse to the equivalence of \cref{bla cor} (by evaluating on $\pi_0$).

To obtain the final statement let $\varphi \colon \L^\n(\Z) \to \L^{\q}(\Z)[1]$ correspond to $1$ under the equivalence and let $F$ be its fibre. Let $k \in \Z$ be such that $k\varphi = \partial$, where $\partial$ is boundary map $\L^{\n}(\Z) \to \L^{\q}(\Z)[1]$. Then we find a commutative diagram
\[ \begin{tikzcd}
	\L^{\q}(\Z) \ar[r] \ar[d,"\cdot k"] & F \ar[r] \ar[d] & \L^{\n}(\Z) \ar[r,"\varphi"] \ar[d,equal] & \L^{\q}(\Z)[1] \ar[d,"\cdot k"] \\
	\L^{\q}(\Z) \ar[r,"\cdot 8"] & \L^{\s}(\Z) \ar[r] & \L^{\n}(\Z) \ar[r,"\partial"] & \L^{\q}(\Z)[1]
\end{tikzcd}\]

Looking at the sequence of homotopy groups in degrees $4k$, we find from the snake lemma that $\pi_{4k}(F) \to \pi_{4k}(\L^{\s}(\Z))$ is injective and has cokernel $\Z/k$. In other words, the map identifies with the multiplication by $k$ on $\Z$, and it follows that $k$ must be congruent to $1$ modulo $8$, which gives the second claim. 
\end{proof}

\begin{Prop}\label{xyz}
The equivalences of \cref{adlslq} and \cref{lnand} extend to an equivalence of the $\L^{\s}(\Z)$-linear fibre sequences
\[\L^{\q}(\Z) \xrightarrow{\mathrm{sym}} \L^{\s}(\Z) \longrightarrow \L^{\n}(\Z)\]
and
\[\I(\L^{\s}(\Z)) \xrightarrow{\I(\mathrm{sym})} \I(\L^{\q}(\Z)) \xrightarrow{\I(\partial)} \I(\L^{\n}(\Z)[-1]).\]
In fact, the space of such extensions forms a torsor for $\Omega^{\infty+1}(\L^{\s}(\Z)\oplus \L^{\n}(\Z))$ and so the indeterminacy of such an identification of fibre sequences is a $(\Z/2)^2$-torsor.
\end{Prop}

We remark, that the $(\Z/2)^2$ appearing as indeterminacy in this statement seems unrelated to those appearing in \cref{splitln} and \cref{symhop}. In particular, it is not difficult to check that it acts trivially on the $\L(\R)$-linear self-homotopies of the symmetrisation map.

\begin{proof}
We first show that the square involving the right hand maps commutes. To this end we compute, that 
\[\map_{\L^{\s}(\Z)}(\I(\L^{\q}(\Z),\L^{\n}(\Z)) \simeq \map_{\L^{\s}(\Z)}(\L^{\s}(\Z),\L^{\n}(\Z)) \simeq \L^{\n}(\Z)\]
via the equivalence of \cref{adlslq}. In particular, such maps are detected by their behaviour on $\pi_0$, where both composites are readily checked to induce the projection $\Z \rightarrow \Z/8$. Choosing a homotopy between the composites, we obtain an equivalence of fibre sequences and it remains to check, that the induced map on fibres is the Anderson dual of the middle one. But we have
\[\map_{\L^{\s}(\Z)}(\I(\L^{\s}(\Z)), \L^{\q}(\Z)) \simeq \map_{\L^{\s}(\Z)}(\L^{\q}(\Z), \L^{\q}(\Z)) \simeq \L^{\s}(\Z)\]
by \cref{endos of Lq}, so again such maps are detected by their effect on components. In the case at hand this is easily checked to be the canonical isomorphism by means of the long exact sequences associated to the fibre sequences. 

To obtain the statement about the automorphisms, recall that endomorphisms of fibre sequences are equally well described by the endomorphisms of their defining arrow. 
By \cite[Proposition 5.1]{GHN} these endomorphisms are given by
\[\Map_{\L^{\s}(\Z)}(\L^{\s}(\Z),\L^{\s}(\Z)) \times_{\Map_{\L^{\s}(\Z)}(\L^{\s}(\Z),\L^{\n}(\Z))} \Map_{\L^{\s}(\Z)}(\L^{\n}(\Z),\L^{\n}(\Z)),\]
in the present situation. We were unable to identify the right hand term, but requiring the endomorphism to be the identity termwise, results in the space
\[\{\id_{\L^{\q}(\Z)}\} \times_{\Map_{\L^{\s}(\Z)}(\L^{\q}(\Z),\L^{\q}(\Z))} \{\id_{\L^{\s}(\Z)}\} \times_{\Map_{\L^{\s}(\Z)}(\L^{\s}(\Z),\L^{\n}(\Z))} \{\id_{\L^{\n}(\Z)}\}\]
which by \cref{endos of Lq} evaluates as claimed.
\end{proof}

\begin{Rmk}
\cref{lnand} and the existence part of \cref{xyz} can also be shown using the following argument.
The equivalences
\[\map_{\L^{\s}(\Z)}(\L^{\q}(\Z),\L^{\s}(\Z)) \simeq \map_{\L^{\s}(\Z)}(\L^{\q}(\Z),\map(\L^{\q}(\Z),\I(\S))) \simeq \map(\L^{\q}(\Z)\otimes_{\L^{\s}(\Z)}\L^{\q}(\Z), \I(\S))\]
give a $\mathrm{C}_2$-action to the left hand spectrum via the flip action on the right. Decoding definitions, the action map on the left hand spectrum is given by applying Anderson duality and then conjugating with the equivalences of \cref{adlslq}. 

The symmetrisation map is a homotopy fixed point for this action, since unwinding definitions shows that it is sent to the composite
\[\L^{\q}(\Z)\otimes_{\L^{\s}(\Z)}\L^{\q}(\Z) \stackrel{\mu}{\lto} \L^{\q}(\Z) \simeq \I(\L^{\s}(\Z)) \longrightarrow   \I(\S),\]
where the final map is induced by the unit of $\L^{\s}(\Z)$; this composite is a homotopy fixed point since the multiplication $\L^{\q}(\Z)$ makes it a non-unital $\E_\infty$-ring spectrum by \cite{HLLII}.
\end{Rmk}

Finally, we complete the proof of \cref{thmA}. Consider the map $u\colon \L^{\s}(\Z) \to \L(\R)$ induced by the evident ring homomorphism and its Anderson dual $\I(u) \colon \L(\R) \to \L^{\q}(\Z)$ arising from \cref{self dual} and \cref{adlslq}.

\begin{Cor}\label{symmetric case}
There exists a unique equivalence
\[\L^{\q}(\Z) \simeq \L(\R) \oplus (\L(\R)/2)[-2]\]
of $\L(\R)$-modules inducing under which the map $\I(u)$ corresponds to the inclusion. In particular, there is a preferred equivalence
\[\L^{\q}(\Z) \simeq \L(\R) \oplus \bigoplus_{k \in \Z} \H\Z/2[4k-2]\]
of underlying spectra.
\end{Cor}

Recall the analogous equivalence 
\[\L^{\s}(\Z) \simeq \L(\R) \oplus (\L(\R)/2)[1]\]
from \cref{lstype}.

\begin{proof}
The existence of such a splitting is for example immediate from that for $\L^{\s}(\Z)$, the self duality of $\L(\R)$ from \cref{self dual} and the equivalence
\[\I(\L(\R)/2) \simeq (\L(\R)/2)[-1]\]
obtained by applying Anderson duality to the fibre sequence defining $\L(\R)/2$.

For the uniqueness we compute
\begin{align*}
\pi_0\mathrm{end}_{\L(\R)}(\L(\R) \oplus (\L(\R)/2)[-2]) =\ &\pi_0\mathrm{end}_{\L(\R)}(\L(\R)) \oplus \pi_0\map_{\L(\R)}(\L(\R), (\L(\R)/2)[-2]) \ \oplus \ \\ &  \pi_0\mathrm{map}_{\L(\R)}((\L(\R)/2)[-2], \L(\R)) \oplus \pi_0\mathrm{end}_{\L(\R)}((\L(\R)/2)[2]) \\
=\ & \Z \oplus 0 \oplus 0 \oplus \Z/2
\end{align*}
using the exact sequences associated to multiplication by $2$. The two factors are again detected by the effect on $\pi_0$ and $\pi_2$, respectively, and thus determined by compatibility with $\I(u)$ and by being an equivalence.

The statement about the underlying spectra is immediate from \cref{drtype}.
\end{proof}

\begin{Prop}\label{symhop}
Under the equivalences 
\[\L^{\q}(\Z) \simeq \L(\R) \oplus (\L(\R)/2)[-2] \quad \L^{\s}(\Z) \simeq \L(\R) \oplus (\L(\R)/2)[1]\]
of \cref{lstype} and \cref{symmetric case} the symmetrisation map is $\L(\R)$-linearly homotopic to the matrix
\[\begin{pmatrix}8 & 0 \\ 0 & 0\end{pmatrix},\]
moreover the $\L(\R)$-linear homotopies form a $(\Z/2)^2$-torsor. Each induces a splitting
\[\L^{\n}(\Z) \simeq \L(\R)/8 \oplus (\L(\R)/2)[1] \oplus (\L(\R)/2)[-1]\]
of $\L(\R)$-modules and the underlying splittings of spectra are those from \cref{splitln}.
\end{Prop}
\begin{proof}
We analyse the components individually. The induced map on the $\L(\R)$-factors is multiplication by $8$ on $\pi_0$ and therefore homotopic to multiplication by $8$ as the source is a free module, and since $\L_1(\R) = 0$, a witnessing homotopy is unique. For the other parts it is a simple computation that all three of 
\[\map_{\L(\R)}(\L(\R),(\L(\R)/2)[1]),\quad  \map_{\L(\R)}((\L(\R)/2)[-2],\L(\R)) \] \[\quad \text{and} \quad  \map_{\L(\R)}((\L(\R)/2)[-2],(\L(\R)/2)[1])\]
have vanishing components. This gives the existence of a homotopy as claimed. 

For the indeterminacy one computes the first homotopy groups of the three terms are given by $\Z/2, \Z/2$ and $0$, respectively. Together with the uniqueness statement for the $\L(\R)$-part this gives the claim.

The comparison to \cref{splitln} follows by simply unwinding the definitions.
\end{proof}

\begin{Rmk}
\begin{enumerate}
\item \cref{symmetric case} can also easily be obtained directly from \cref{lzlrsplit}: One considers the map $v \colon \L(\R) \rightarrow \L^{\q}(\Z)$ obtained by extending the generator $\S \rightarrow \L^{\q}(\Z)$ of $\L_0^\q(\Z)$, and observes that the resulting fibre sequence splits $\L(\R)$-linearly, see \cref{generalsplits}. The underlying equivalence of spectra can also be obtained even more directly from the $2$-local fracture square of $\L^{\q}(\Z)$.
\item With such an independent argument \cref{symmetric case} can then be used in conjunction with \cref{lstype} to deduce that $\L^{\q}(\Z)$ and $\L^{\s}(\Z)$ are Anderson dual to one another, however, only as $\L(\R)$- and not as $\L^{\s}(\Z)$-modules. Similarly, using this approach it is not clear that $v = \I(u)$ is in fact $\L^{\s}(\Z)$-linear.
\item 
By \cref{symhop} the composite
\[\L(\R) \stackrel{\I(u)}{\lto} \L^{\q}(\Z) \stackrel{\mathrm{sym}}{\lto} \L^{\s}(\Z) \stackrel{u}{\lto} \L(\R)\]
is $\L(\R)$-linearly homotopic to multiplication by $8$. However, by construction this map is $\L^{\s}(\Z)$-linear. We do not know whether it is homotopic to multiplication by $8$ as such, largely since we do not understand $\L(\R)$ as an $\L^{\s}(\Z)$-module, compare to our question in \cref{question}.

\end{enumerate}
\end{Rmk}

As a final application of the results above, we consider the space $\G/\Top$. It is the fibre of the topological $J$-homomorphisms $\B\Top \to \B\G$ and is therefore endowed with a canonical $\E_\infty$-structure usually referred to as the Whitney-sum structure. As mentioned in the introduction one of the principal results of surgery theory is an equivalence of \emph{spaces} $\G/\Top \to \Omega^\infty_0\L^{\q}(\Z)$; see \cite[Theorem C.1]{KS}. We therefore find:

\begin{Cor}
There is a preferred equivalence of spaces 
\[ \G/\Top \simeq \Omega^\infty_0 \L(\R) \times \prod\limits_{k \geq 0} \mathrm K(\Z/2,4k+2).\]
\end{Cor}
Let us emphatically warn the reader that the equivalence $\G/\Top \to \Omega^\infty_0 \L^{\q}(\Z)$ is not one of $\E_\infty$-spaces or even H-spaces, when equipping $\G/\Top$ with the Whitney-sum structure. By computations of Madsen \cite{MadsenGTop}, $\B^3(\G/\Top)_{(2)}$ is not even an Eilenberg-Mac Lane space, in contrast to $\Omega^{\infty-3}\L^{\q}(\Z)_{(2)}$. 
%As a consequence of the work in \cite{HLLI}, o
One can nevertheless describe the Whitney-sum $\E_\infty$-structure on $\G/\Top$ purely in terms of the algebraic L-theory spectrum $\L^{\q}(\Z)$: one finds an equivalence of $\E_\infty$-spaces $\G/\Top \simeq \Sl_1\L^{\q}(\Z)$, see the forthcoming \cite{HLLII}.

\section{The homotopy type of $\L^{\gs}(\Z)$}\label{sectiongs}

To address the case of the genuine $\L$-spectra, we need to briefly recall Lurie's setup for $\L$-spectra from \cite{LurieL}, which forms the basis for their construction and analysis in the series \cite{CDHI, CDHII, CDHIII, CDHIV}. In the terminology of those papers the input for an $\L$-spectrum is a Poincar\'e $\infty$-category, which is a small stable $\infty$-category $\C$ equipped with a quadratic functor $\QF \colon \C^\mathrm{op} \to \Sp$, satisfying a non-degeneracy condition. These objects form an $\infty$-category $\Catp$ and Lurie constructs a functor
\[\L \colon \Catp \longrightarrow \Sp,\]
see \cite[Section 4.4]{CDHII}. The examples considered before, that is symmetric and quadratic $\L$-spectra of a ring with involution $R$, arise by considering $\C = \Dperf(R)$, the category of perfect complexes over $R$, equipped with the quadratic functors
\[\QF^s(C) = \map_{R \otimes R}(C \otimes C,R)^{\mathrm{hC}_2}\quad \text{and} \quad \QF^q(C) = \map_{R \otimes R}(C \otimes C,R)_{\mathrm{hC}_2},\]
respectively, where $\mathrm{C}_2$ acts by the involution on $R$, by flipping the factors on $C \otimes C$ and through conjugation on the mapping spectrum.

Now, the animation\footnote{Here, we use novel terminology suggested by Clausen and first introduced in \cite[5.1.4]{CS} } (in more classical terminology the non-abelian derived functor) 
of any quadratic functor
$\Lambda \colon \Proj(R)^\mathrm{op} \to \mathrm{Ab}$
gives rise to a genuine quadratic functor 
\[\QF^{\gl} \colon \Dperf(R)^\mathrm{op} \longrightarrow \Sp\]
and we set
\[\L^{\gl}(R) = \L(\Dperf(R),\QF^{\gl}),\]
whenever $(\Dperf(R),\QF^{\gl})$ is a Poincar\'e category (which can be read off from a non-degeneracy condition on $\Lambda$).
This can be applied for example to the functors
\[P \longmapsto \Hom_{R \otimes R}(P \otimes P,R)^{\mathrm{C}_2} \quad \text{and} \quad P \longmapsto \Hom_{R \otimes R}(P \otimes P,R)_{\mathrm{C}_2}\]
giving Poincar\'e structures $\QF^{\gs}$ and $\QF^\gq$ on $\Dperf(R)$ which are studied in detail in \cite[Section 4.2]{CDHI}. From the universal property of homotopy orbits and fixed points, there are then comparison maps
\[\QF^q \Longrightarrow \QF^\gq \Longrightarrow \QF^\gs \Longrightarrow \QF^s\]
giving rise to maps
\[\L^{\q}(R) \longrightarrow \L^\gq(R) \longrightarrow \L^\gs(R) \longrightarrow \L^{\s}(R),\]
whose composition is the symmetrisation map considered before. All of these are equivalences if $2$ is invertible in $R$, but in general these are four different spectra. An entirely analogous definition gives the skew-symmetric variants by introducing a sign into the action of $\mathrm{C}_2$. 

The final formal property of this version of the $\L$-functor that we need to recall is that it admits a canonical lax symmetric monoidal refinement: By the results of \cite[Section 5.2]{CDHI} the category $\Catp$ admits a symmetric monoidal structure that lifts Lurie's tensor product on $\mathrm{Cat}^\mathrm{ex}_\infty$, the category of stable categories. In \cite{CDHIV} we then show that the functor $\L \colon \Catp \rightarrow \Sp$ admits a canonical lax symmetric refinement and for a commutative ring with trivial involution the functors $\QF^s$ and $\QF^\gs$ enhance $\Dperf(R)$ with its tensor product to an $\E_\infty$-object of $\Catp$. Furthermore, the forgetful transformation $\QF^\gs \Rightarrow \QF^s$ refines to an $\E_\infty$-map and $(\Dperf(R),\QF^q)$ and $(\Dperf(R),\QF^\gq)$ admit canonical module structures over their respective symmetric counterparts, such that the forgetful map $\QF^q \Rightarrow \QF^\gq$ becomes $\QF^\gs$-linear, see \cite[Section 5.4]{CDHI}. By the monoidality of the $\L$-functor these structures persist to the $\L$-spectra. However, neither $\L^\gs(R)$ nor $\L^\gq(R)$ are generally modules over $\L^{\s}(R)$.
Combining work of Ranicki with the main results of \cite{CDHIII} we obtain:

\begin{Thm}\label{homotopylgs}
Of the comparison maps
\[\L^{\q}(\Z) \longrightarrow \L^\gq(\Z) \longrightarrow \L^\gs(\Z) \longrightarrow \L^{\s}(\Z)\]
the left, middle, and right one induce isomorphisms on homotopy groups below degree $2$, outside degrees $[-2,1]$, and in degrees $0$ and above, respectively. Furthermore, the preimage of the element $x \in \L_4^\s(\Z)$ in $\L_4^\gq(\Z)$ determines an equivalence
\[\L^\gs(\Z)[4] \longrightarrow \L^\gq(\Z).\]
\end{Thm}

In particular, the homotopy groups of $\L^\gs(\Z)$ evaluate to
\[\L_n^\gs(\Z) = \begin{cases} \Z & n \equiv 0 \ (4) \\
                               \Z/2 & n \equiv 1 \ (4) \text{ and } n \geq 0 \\
                               \Z/2 & n \equiv 2 \ (4) \text{ and } n \leq -4 \\
                               0 & \text{else.}\end{cases}\]

\begin{Cor}\label{truncsym}
The symmetrisation map determines a cartesian square
\[\begin{tikzcd}
	\L^\gs(\Z) \ar[r] \ar[d] & \L^{\s}(\Z) \ar[d] \\
	\tau_{\geq -1} \L^{\n}(\Z) \ar[r] & \L^{\n}(\Z)
\end{tikzcd}\]
of $\L^\gs(\Z)$-modules and the horizontal maps are localisations at $x \in \L_4^{\gs}(\Z)$.
\end{Cor}

Here the module structure on the lower left corner is induced by the fibre sequence
\[\L^{\q}(\Z) \longrightarrow \L^\gs(\Z) \longrightarrow \tau_{\geq -1} \L^{\n}(\Z), \]
which also proves the Corollary.

\begin{Rmk}\label{weirdring}
\begin{enumerate}
\item 
The homotopy ring of $\L^\gs(\Z)$ is rather complicated to spell out in full. It is entirely determined by the assertion that the multiplication with $x \in \L_4^\gs(\Z)$ is an isomorphism in all degrees in which it can, with one exception: The map $\cdot x \colon \L_{-4}^\gs(\Z) \to \L_0^\gs(\Z)$ is multiplication by 8.

Using the evident generators, this results in $\L_*^\gs(\Z) = \Z[x,e,y_i,z_i | i \geq 1]/I$ with $|x| = 4, |e| = 1, |y_i| = -4i$ and $|z_i| = -4i-2$, where $I$ is the ideal spanned by the elements
\[2e,e^2, ey_i,ez_i,xy_1 - 8, xy_{i+1} - y_i, xz_1, xz_{i+1} - z_i, y_iy_j - 8y_{i+j},y_iz_j, 2z_i, z_iz_j.\]
In particular, $\L_*^\gs(\Z)$ is not finitely generated as an algebra over $\Z$ or even $\L^\gs_{\geq 0}(\Z) = \Z[x,e]/(2e,e^2)$.
\item One can show that the square in \cref{truncsym} is in fact one of $\E_\infty$-ring spectra, which is somewhat surprising for the lower left term. To see this, one has to show that the maps $\Qoppa^\q \Rightarrow \Qoppa^{\gs} \Rightarrow \Qoppa^\s$ both exhibit the source as an ideal object of the target with respect to the  tensor product of quadratic functors established in \cite[Section 5.3]{CDHI}. Then one can use the monoidality of $\L$-theory. We refrain from giving the details here.
\end{enumerate}
\end{Rmk}

We begin with the proof of \cref{thmB}:

\begin{Thm}\label{AD for genuine L}
The Anderson dual of $\L^{\gs}(\Z)$ is given by $\L^{\gq}(\Z)$, or equivalently by $\L^{\gs}(\Z)[4]$, as an $\L^{\gs}(\Z)$-module.
\end{Thm}

\begin{proof}
Consider the fibre sequence
\[ \L^{\gs}(\Z) \lto \L^{\s}(\Z) \lto \tau_{\leq -3}\L^{\n}(\Z)\]
from \cref{truncsym} and observe that \cref{lnand} provides a canonical equivalence 
\[\I(\tau_{\leq -3}\L^{\n}(\Z)) \simeq (\tau_{\geq 3}\L^{\n}(\Z))[-1]\]
of $\L^\gs(\Z)$-modules. Rotating the above fibre sequence once to the left and applying Anderson duality in combination with \cref{adlslq}, we obtain a fibre sequence
\[ \L^{\q}(\Z) \lto \I(\L^{\gs}(\Z)) \lto \tau_{\geq 3}\L^{\n}(\Z).\]
From \cref{truncsym} we also have a fibre sequence
\[ \L^{\q}(\Z) \lto \L^{\gs}(\Z)[4] \lto \tau_{\geq 3}\L^{\n}(\Z)\]
and both are sequences of $\L^{\gs}(\Z)$-modules. Now 
\[\pi_4\I(\L^{\gs}(\Z)) \cong \Hom(\L^{\gs}_{-4}(\Z),\Z) \cong \Z\]
with the $E_8$-form providing a generator. It determines an $\L^\gs(\Z)$-module map
\begin{equation}\label{map-to-be-equivalence}
\L^\gs(\Z)[4] \longrightarrow  \I(\L^{\gs}(\Z)) \tag{$\ast$}
\end{equation}
which we wish to show is an equivalence. To see this consider the diagram 
\[\begin{tikzcd}
	\L^{\q}(\Z) \ar[r] \ar[d,equal] & \I(\L^{\gs}(\Z)) \ar[r] & \tau_{\geq 3}\L^{\n}(\Z) \ar[d, equal] \\
	\L^{\q}(\Z) \ar[r] & \L^{\gs}(\Z)[4] \ar[r] \ar[u] & \tau_{\geq 3}\L^{\n}(\Z)
\end{tikzcd}
\]
which we claim commutes up to homotopy and consists of horizontal fibre sequences; let us mention explicitly that we do not claim here that it is a diagram of fibre sequences, which would require us to identify the witnessing homotopies; although this follows a posteriori, it is not necessary for the argument. The commutativity of the above diagram on homotopy groups  implies that the map \eqref{map-to-be-equivalence} is indeed an equivalence by a simple diagram chase.

We now establish the claim that the above diagram commutes up to homotopy.
We discuss first the right square. Since the source of the composites is a free $\L^{\gs}(\Z)$-module in this case, it suffices to show that this diagram commutes on $\pi_4$. By construction, both composites induce the projection $\Z \rightarrow \Z/8$. 
For the left hand square, we use the calculation proving \cref{bla cor} together with \cref{adlslq} to obtain an equivalence
\[ \map_{\L^{\gs}(\Z)}(\L^{\q}(\Z),\I(\L^{\gs}(\Z))) \simeq \I(\L^{\q}(\Z)) \simeq \L^{\s}(\Z).\]
In particular, $\pi_0$ of this space is $\Z$, and clearly we can distinguish all the possible maps by their effect on $\pi_{4}$, as there is one that is non-zero. It follows that it again suffices to argue that the diagram commutes on $\pi_{4}$. In this case, both maps are the multiplication by $8$, and so are homotopic as needed.
\end{proof}

\begin{Rmk}
One can also hope to prove \cref{AD for genuine L} directly using \cref{free homotopy implies free module}, instead of reducing it to \cref{adlslq}. However, we did not manage to compute the module structure of $\I(\L^\gs(\Z))$ over $\L^\gs(\Z)$ without referring to \cref{adlslq}: Everything can be formally reduced to the statement that multiplication by the element in $z_1 \in \L^\gs_{-6}(\Z)$ induces a surjection $\pi_4(\I(\L^\gs(\Z))) \rightarrow \pi_{-2}(\I(\L^\gs(\Z)))$, but we did not find a direct argument for this. 

If one had a direct argument for \cref{AD for genuine L} one could conversely deduce \cref{adlslq} as there are equivalences
\[\L^{\s}(\Z) \simeq \colim\limits_{\cdot x} \L^{\gs}(\Z) \quad \text{and}\quad \L^{\q}(\Z) \simeq \lim\limits_{\cdot x} \L^{\gq}(\Z)\]
as a result of \cref{homotopylgs}.
\end{Rmk}

\begin{Cor}\label{canonical maps homotopic}
Under the equivalences of \cref{adlslq} and \cref{AD for genuine L}, the Anderson dual of the canonical map $\L^\gs(\Z) \rightarrow \L^{\s}(\Z)$ is homotopic to the canonical map $\L^{\q}(\Z) \rightarrow \L^\gq(\Z)$. In fact, the $\L^{\gs}(\Z)$-linear homotopies between these maps form a $\Z/2$-torsor.
\end{Cor}

As the symmetrisation map $\mathrm{sym} \colon \L^{\q}(\Z) \rightarrow \L^{\s}(\Z)$ factors through the above maps and the canonical one $\L^\gq(\Z) \rightarrow \L^\gs(\Z)$, which \cref{homotopylgs} identifies with the multiplication by $x$ on $\L^\gs(\Z)$, the present corollary provides another proof that $\I(\mathrm{sym}) \simeq \mathrm{sym}$ under the equivalences provided by \cref{adlslq}.

\begin{proof}
By duality we have equivalences
\[\map_{\L^\gs(\Z)}(\L^{\q}(\Z), \L^\gq(\Z)) \simeq \map_{\L^\gs(\Z)}(\L^\gs(\Z), \L^{\s}(\Z)) \simeq \L^{\s}(\Z),\]
so such maps are detected on homotopy. Both claims are now immediate.
\end{proof}

Finally, we address the second half of \cref{thmB} and split $\L^{\gs}(\Z)$ into a torsionfree part and a torsion part. For the precise statement, we put $\ell(\R) = \tau_{\geq 0}\L(\R)$ and note that the equivalence $\tau_{\geq 0} \L^\gs(\Z) \to \tau_{\geq 0} \L^{\s}(\Z)$ together with \cref{lzlrsplit} provides an $\E_1$-map $\ell(\R) \to \L^\gs(\Z)$ splitting the canonical map on connective covers. We will use this map to regard $\L^\gs(\Z)$ as an $\ell(\R)$-module. 

The torsionfree part of $\L^\gs(\Z)$ is given by the following spectrum:

\begin{Def}
Define the $\ell(\R)$-module $\LLL$ by the cartesian square
\[\begin{tikzcd}
	\LLL \ar[r] \ar[d] &  \L(\R)\ar[d] \\
	\ell(\R)/8 \ar[r] & \L(\R)/8.
\end{tikzcd}\]
\end{Def}             

One easily checks that $\pi_{4k}\LLL \cong \Z$ for all $k \in \Z$, whereas all other homotopy groups vanish.

\begin{Rmk}\label{Tisweird}
The diagram defining $\LLL$ in fact refines to one of $\E_2$-ring spectra: The equivalence
\[\ell(\R)/8 \simeq \ell(\R)_{(2)} \otimes_{\H\Z} \H\mathbb{Z}/8.\]
from \cref{waystosplit} exhibits the source as an $\E_2$-ring by \cref{E3}, and similarly for the periodic case. The homotopy ring of $\LLL$ is then described as
 $\Z[x] \oplus_{8\Z[x]} 8\Z[x^{\pm1}] \subset \Z[x^{\pm 1}]$.

Contrary to the case of $\L^{\s}(\Z)$ and $\L(\R)$, we do not, however, know of sensible ring maps connecting $\L^\gs(\Z)$ and $\LLL$: An $\E_1$-map $\L^\gs(\Z) \rightarrow \LLL$ cannot be particularly compatible with the squares
\[\begin{tikzcd}
	\LLL \ar[r] \ar[d] &  \L(\R)\ar[d]  & & \L^\gs(\Z) \ar[r] \ar[d] & \L^{\s}(\Z) \ar[d] \\
	\ell(\R)/8 \ar[r] & \L(\R)/8 & & \tau_{\geq -1} \L^{\n}(\Z) \ar[r] & \L^{\n}(\Z),
\end{tikzcd}\]
since by \cref{ringln} no ring map $\L^{\n}(\Z) \rightarrow \L(\R)/8$ can induce an isomorphism on $\pi_4$. Conversely, to construct a ring map $\LLL \rightarrow \L^\gs(\Z)$ from these squares, one would need a ring map $\L(\R)/8 \rightarrow \L^{\n}(\Z)$, and thus in turn an $\E_1$-factorisation of $\H\Z \rightarrow \L^{\n}(\Z)$ through $\H\Z/8$. We do not know whether such a factorisation exists. If it does, \cref{abc} below could be upgraded to an $\LLL$-linear splitting of $\L^\gs(\Z)$.
\end{Rmk}

We are ready to finish the proof of \cref{thmB}:

\begin{Thm}\label{abc}
There is a unique equivalence
\[ \L^{\gs}(\Z) \simeq \LLL \oplus (\ell(\R)/2)[1] \oplus (\L(\R)/(\ell(\R),2))[-2],\]
of $\ell(\R)$-modules, such that the canonical map $\L^\gs(\Z) \rightarrow \L(\R)$ corresponds to the projection to the first summand followed by the canonical map $\LLL \rightarrow \L(\R)$.  

In particular, there is a preferred equivalence
\[ \L^{\gs}(\Z) \simeq \LLL \oplus \bigoplus_{k \geq 0} \H\Z/2[4k+1] \oplus \H\Z/2[-4k-6]\]
of underlying spectra.
\end{Thm}

We can also identify the induced maps involving the classical $\L$-spectra of the integers. For the statement recall from \cref{lstype} and \cref{symmetric case} that there are canonical equivalences
\[\L^{\s}(\Z) \simeq \L(\R) \oplus (\L(\R)/2)[1] \quad \text{and} \quad \L^{\q}(\Z) \simeq \L(\R) \oplus (\L(\R)/2)[2]\]
of $\L(\R)$-modules. Recall also that $\L^\gq(\Z) \simeq \L^\gs(\Z)[4]$ via the multiplication by $x$.

\begin{add}\label{def}
Under these three equivalences the $\ell(\R)$-linear homotopies between the canonical maps
\[\L^{\q}(\Z) \longrightarrow \L^\gq(\Z) \quad \text{and}\quad \L^\gs(\Z) \longrightarrow \L^{\s}(\Z)\]
and the maps represented by the matrices
\[\begin{pmatrix} \I(\can) &0 \\ 0&0 \\0 &\mathrm{pr} \end{pmatrix} \quad \text{and}\quad \begin{pmatrix} \can & 0&0 \\ 0&\mathrm{incl} &0 \end{pmatrix},\]
form a $\Z/2$-torsor each; here $\can \colon \LLL \rightarrow \L(\R)$ is the map from the definition of $\LLL$, and 
\[\I(\can) \colon \L(\R) \simeq \I(\L(\R)) \rightarrow \I(\LLL) \simeq \LLL[4]\]
is its Anderson dual, uniquely determined from the map $\ell(\R) \rightarrow \LLL[4]$ arising from the generator $8x^{-1} \in \pi_{-4}(\LLL)$.
\end{add}

The canonical map $\L^\gq(\Z) \rightarrow \L^\gs(\Z)$, being identified with multiplication by $x \in \ell_4(\R)$ on $\L^\gs(\Z)$, is compatible with the splitting of the theorem. Composing with it, the four $\ell(\R)$-linear homotopies of the addendum necessarily underlie the four $\L(\R)$-linear homotopies of \cref{symhop}.

\begin{proof}[Proof of \cref{abc} \& \cref{def}]
For the existence part, we first claim that the canonical map $\L^\gs(\Z) \rightarrow \L(\R)$ factors through $\LLL$. From the definition we have a fibre sequence
\[\LLL \longrightarrow \L(\R) \longrightarrow \L(\R)/(\ell(\R),8),\] so we need to check that the composite 
\[\L^\gs(\Z) \longrightarrow \L(\R) \longrightarrow \L(\R)/(\ell(\R),8)\]
is null homotopic. Note that it induces the zero map on homotopy groups by \cref{homotopylgs}. We will argue that this implies the claim. Using \cref{AD for genuine L} we find
\begin{align*}
\map_{\ell(\R)}(\L^\gs(\Z), \L(\R)/(\ell(\R),8)) &\simeq \map_{\ell(\R)}(\L^\gs(\Z),\I((\ell(\R)/8)[3])) \\
&\simeq \map_{\ell(\R)}((\ell(\R)/8)[3],\L^\gs(\Z)[4]).
\end{align*}
The arising fibre sequence
\[ \L^\gs(\Z) \xrightarrow{\cdot 8} \L^\gs(\Z) \longrightarrow \map_{\ell(\R)}(\L^\gs(\Z), \L(\R)/(\ell(\R),8)) \]
gives an isomorphism
\[\Z/8 \cong \pi_0\map_{\ell(\R)}(\L^\gs(\Z), \L(\R)/(\ell(\R),8)),\]
which is necessarily detected on $\pi_{-4}$, as the composite
\[\tau_{\leq -4} \L^\gs(\Z) \simeq \tau_{\leq -4} \L^{\q}(\Z) \longrightarrow \tau_{\leq -4} \L(\R) \longrightarrow \tau_{\leq -4}\L(\R)/8 \simeq \L(\R)/(\ell(\R),8),\]
with the middle map the projection arising from \cref{symmetric case}, is clearly $\ell(\R)$-linear and induces the projection $\Z \rightarrow \Z/8$ on $\pi_{-4}$. By \cref{homotopylgs} any lift $\L^\gs(\Z) \rightarrow \LLL$ induces an isomorphism the torsionfree part of the homotopy groups of $\L^\gs(\Z)$. 

We next produce the map $\L^\gs(\Z) \rightarrow (\ell(\R)/2)[1]$. We claim that the composite 
\[\L^\gs(\Z) \longrightarrow \L^{\s}(\Z) \longrightarrow (\L(\R)/2)[1],\]
factors as needed, where the second map comes from the splitting \cref{lstype}. To this end we compute just as before that
\begin{align*}
\map_{\ell(\R)}(\L^\gs(\Z), (\L(\R)/(\ell(\R),2))[1]) &\simeq \map_{\ell(\R)}(\L^\gs(\Z),\I((\ell(\R)/2)[2])) \\
&\simeq \map_{\ell(\R)}((\ell(\R)/2)[2],\L^\gs(\Z)[4]).
\end{align*}
This results in a fibre sequence
\[ \map_{\ell(\R)}(\L^\gs(\Z), \L(\R)/(\ell(\R),2)[1]) \longrightarrow  \L^\gs(\Z)[2] \xrightarrow{\cdot 2} \L^\gs(\Z)[2]\]
and thus 
\[0 \cong \pi_0\map_{\ell(\R)}(\L^\gs(\Z), \L(\R)/(\ell(\R),8)),\]
giving us a lift $\L^\gs(\Z) \rightarrow (\ell(\R)/2)[1]$, as desired. By construction it induces an isomorphism of the positive degree torsion part of the homotopy groups of $\L^\gs(\Z)$

The final map $\L^\gs(\Z) \rightarrow (\L(\R)/(\ell(\R),2))[-2]$ is easier. It arises by expanding the composite
\[\tau_{\leq -6} \L^\gs(\Z) \simeq \tau_{\leq -6} \L^{\q}(\Z) \longrightarrow \tau_{\leq -6} \L(\R)[2] \simeq (\L(\R)/(\ell(\R),2))[-2]\]
to all of $\L^\gs(\Z)$ (by coconnectivity of the target), where the second map arises from the splitting of \cref{symmetric case}. It induces an equivalence on the negative degree torsion part of the homotopy groups of $\L^\gs(\Z)$ by construction. 

Combining these three maps, gives the existence part of the theorem. \\

Next we prove the existence homotopies as in the addendum. To start note that the claims about the first two columns of the second matrix are true by construction. For the last column simply observe that the mapping space $\map_{\ell(\R)}(\L(\R)/(\ell(\R),2),M)$ is contractible for any $\L(\R)$-module $M$, since the source is $x$-torsion, whereas $x$ is invertible in the target. The claim about the first matrix is then immediate from \ref{canonical maps homotopic}. \\

To obtain uniqueness of the splitting, we again treat all parts separately. From the fibre sequence defining $\LLL$ we find a fibre sequence
\[\map_{\ell(\R)}(\LLL,\LLL) \longrightarrow \map_{\ell(\R)}(\LLL,\L(\R)) \longrightarrow \map_{\ell(\R)}(\LLL,\L(\R)/(\ell(\R),8))\]
Since $T[x^{-1}] \simeq \L(\R)$, the middle term evaluates to $\L(\R)$ and for the latter easy manipulation using \cref{self dual} give $\I(\LLL) \simeq \LLL[4]$ and $\I(\L(\R)/(\ell(\R),8)) \simeq (\ell(\R)/8)[3]$, so
\[\map_{\ell(\R)}(\LLL,\L(\R)/(\ell(\R),8)) \simeq \map_{\ell(\R)}(\ell(\R)/8,\LLL[1]).\]
From the fibre sequence defining the source of the latter term we find 
\[\pi_1(\map_{\ell(\R)}(\LLL,\L(\R)/(\ell(\R),8)) = 0\] 
so $\pi_0\map_{\ell(\R)}(\LLL,\LLL)$ injects into the integers (with index $8$, but we do not need this), so is detected on homotopy. 

Secondly, we have 
\[\map_{\ell(\R)}(\LLL,(\ell(\R)/2)[1]) \simeq \map_{\ell(\R)}(\L(\R)/(\ell(\R),8),\ell(\R)/2) \simeq \map_{\ell(\R)}(\ell(\R)/8,\ell(\R)/2)[-1]\]
by noting that for any connective $\ell(\R)$-module $M$ we have
\begin{align*}
\map_{\ell(\R)}(\L(\R)/k,M) &\simeq \map_{\ell(\R)}(\L(\R),M/k)[-1] \\
&\simeq \map_{\ell(\R)}(\colim_{\cdot x}\ell(\R), M/k)[-1] \\
&\simeq \lim_{\cdot x} M/k[-1] \\&
\simeq 0
\end{align*}
and applying this to the fibre sequences defining the sources of the left to terms. The arising fibre sequence
\[\map_{\ell(\R)}(\LLL,\ell(\R)/2[1]) \longrightarrow (\ell(\R)/2)[-1] \xrightarrow{\cdot 8} (\ell(\R)/2)[-1]\]
shows that $\pi_0\map_{\ell(\R)}(\LLL,\ell(\R)/2[1])$ vanishes.

Thirdly, by duality we have 
\[\map_{\ell(\R)}(\LLL,\L(\R)/(\ell(\R),2)[-2]) \simeq \map_{\ell(\R)}((\ell(\R)/2), \LLL)[-1],\]
so from the long exact sequence associated to multiplication by $2$ we again find
\[\pi_0\map_{\ell(\R)}(\LLL,\L(\R)/(\ell(\R),2)[-2]) = 0.\]

Next up, both $\pi_0\map_{\ell(\R)}(\ell(\R)/2[1],\LLL)$ and $\pi_0\map_{\ell(\R)}(\ell(\R)/2[1],\L(\R)/(\ell(\R),2)[-2])$ vanish again by the long exact sequence of multiplication by $2$, whereas it gives
\[\map_{\ell(\R)}(\ell(\R)/2[1],\ell(\R)/2[1]) = \Z/2\]
clearly detected on homotopy.

Finally, for $M$ an $\ell(\R)$-module we find
$\map_{\ell(\R)}((\L(\R)/(\ell(\R),2))[-2],M)$ to be the total homotopy fibre of
\[\begin{tikzcd}
	\lim\limits_{\cdot x} M[2] \ar[r] \ar[d,"{\cdot 2}"] & M[2] \ar[d,"{\cdot 2}"] \\
	\lim\limits_{\cdot x} M[2] \ar[r] & M[2]
\end{tikzcd}\]
via the fibre sequence defining the source. For the three modules $\LLL, (\ell(\R)/2)[1]$ and $(\L(\R)/(\ell(\R),2))[-2]$ the components of this total fibre are given by $0,0$ and $\Z/2$, respectively, the last term clearly detected on homotopy again.

Putting these nine calculations together gives the uniqueness assertion of the theorem. \\

Finally, by \cref{canonical maps homotopic} it suffices to verify the (non)-uniqueness assertion of the addendum in the case of the map $\L^\gs(\Z) \rightarrow \L^{\s}(\Z)$. So we need to compute the first homotopy group of
\[\map_{\ell(\R)}(\LLL \oplus (\ell(\R)/2)[1] \oplus (\L(\R)/(\ell(\R),2))[-2], \L(\R) \oplus (\L(\R)/2)[1]) \simeq \mathrm{end}_{\L(\R)}(\L(\R) \oplus (\L(\R)/2)[1]),\]
the equivalence given by inverting $x \in \L_4(\R)$. The four components of the latter are given by
\[\L(\R),\ (\L(\R)/2)[1],\ \fib(\L(\R) \xrightarrow{\cdot 2} \L(\R))[-1] \quad \text{and} \quad \fib(\L(\R)/2 \xrightarrow{\cdot 2} \L(\R)/2)\]
whose first homotopy groups are $0,\Z/2,0$ and $0$, respectively.
\end{proof}

\begin{Rmk}
As mentioned in the introduction to the present section, there are also skew-symmetric and skew-quadratic versions of the genuine $\L$-spectra, and the same is true for the non-genuine spectra as well. However, in the case of the integers it turns out that the sequence
\[\L^{-\q}(\Z) \longrightarrow \L^{-\gq}(\Z) \longrightarrow \L^{-\gs}(\Z) \longrightarrow \L^{-\s}(\Z)\]
is merely a two-fold shift of the sequence considered in the present section. The equivalence is classical for the outer two terms and enters the proof of the $4$-fold periodicity of $\L^{\s}(\Z)$ and $\L^{\q}(\Z)$. The remaining identifications are explained for example around \cite[R.10 Corollary]{CDHIII}.
Note, in particular, that $\L^{-\gs}(\Z)$ is Anderson self-dual.
\end{Rmk}

\appendix
\section{The homotopy ring of $\L^{\n}(\Z)$}

The goal of this appendix is to give a proof of \cref{ringln}, which enters into the proof of \cref{adlslq}. This is a folklore result and is for instance contained in \cite[Equation 1.15]{TW}. However, we had some difficulties extracting a  proof from the literature. The proof we give here was explained to us by Andrew Ranicki and Michael Weiss. We recall the statement:

\begin{Prop}\label{Prop_homotopyL}
The ring structure of normal $\L$-theory is given by 
\[ \pi_*(\L^{\n}(\Z)) \cong \Z/8[x^{\pm 1},e,f]/\Big( 2e=2f=0, e^2=f^2 =0, ef= 4 \Big),\]
where $|e| = 1$ and $|f| = -1$. 
\end{Prop}
\begin{proof}
The only part not immediate from \cref{L(Z)} is that $ef = 4$ in $\L^{\n}_0(\Z)$.  Recall that elements in $\L^{\n}_{k+1}(\Z)$ can be represented by a $-k$-dimensional quadratic Poincar\'e complex equipped with a symmetric null-cobordism.  
\cref{L(Z)} then affords the following description of the generators $e$ and $f$: $e$ is represented by the (closed) symmetric Poincar\'e chain complex $E=\Z/2[-1]$ with its unique non-trivial $-1$-dimensional symmetric form. The element $f$ can be represented by the $2$-dimensional quadratic Poincar\'e chain complex $F=\Z^2[1]$ equipped with the Arf invariant $1$ form $(a,b) \mapsto a^2+b^2+ab \in \mathbb Z/2$, whose underlying anti-symmetric form $((a,b),(c,d)) \mapsto ac-bd$ admits the diagonally embedded $\Z[1]$ as a symmetric Lagrangian $L$. Now, tensoring these forms we obtain a $1$-dimensional quadratic Poincar\'e chain complex $E \otimes F$, equipped with the symmetric Lagrangian $E \otimes L$, the pair of which represent $ef$. 

The direct route to showing that this element is non-trivial would proceed by constructing a quadratic null-bordism of $E \otimes F$, glueing it to $E \otimes L$ and taking the signature of the resulting symmetric bilinear form on $\mathrm{H_{0}}$. Constructing such a null-bordism seems difficult, however, so we will proceed more indirectly. Let us nevertheless note, that $E \otimes L$ becomes acyclic after inverting $2$, so does not contribute to the signature of the glued complex. In particular, the vanishing of $ef$ is determinable purely from $E \otimes F$. 

Consider then the boundary maps $\delta$ associated to the localisation sequences \`a la \cite[Section 3]{phonebook} or \cite[Lectures 8 \& 9]{LurieL} for the morphism $\Z \rightarrow \Z\mathopen{}\left[\tfrac 1 2\right]\mathclose{}$ and the commutative square
\[\begin{tikzcd}
	\L^{\n}(\Z,2)[-1] \ar[r,"\delta"] \ar[d,"\partial"] &\L^{\n}(\Z) \ar[d,"\partial"] \\
         \L^{\q}(\Z,2) \ar[r,"\delta"] & \L^{\q}(\Z)[1]
 \end{tikzcd}\]  
of $\L^{\s}(\Z)$-module spectra containing them horizontally. Since the symmetrisation map $\L^{\q}\big(\Z\mathopen{}\left[\tfrac 1 2\right]\mathclose{}\big) \to \L^{\s}\big(\Z\mathopen{}\left[\tfrac 1 2\right]\mathclose{}\big)$ is an equivalence, we find $\L^{\n}\big(\Z\mathopen{}\left[\tfrac 1 2\right]\mathclose{}\big) \simeq 0$. Thus the upper horizontal map is an equivalence. Lifting the elements along it will therefore suffice to show $0 \neq \partial(e \cdot f) \in \L^{\q}_0(\Z,2)$. It is readily checked from the definitions that this element is represented by $E \otimes F$, considered as a $-1$-dimensional quadratic Poincar\'e chain complex over $\Z$ represented in degrees $[-1,0]$, that becomes acyclic after inverting $2$. 
To such an object $C$ Ranicki attaches in \cite[Proposition 3.4.1]{phonebook} a quadratic linking form as follows: 
Represent the quadratic structure of $C$ as a $0$-cycle $\psi$ of $W \otimes_{\Z[\mathrm C_2]} \Hom(C \otimes C,\Z[1])$, where $W$ is the canonical free resolution of $\Z$ as a $\Z[\mathrm C_2]$-module. Then considering the coefficients $\psi_i \in \Hom(C \otimes C,\Z[1])_{-i}$ against the basis element of $W_i$, he defines
\[\mu \colon \H_{0}(C) \longrightarrow \Z\mathopen{}\left[\tfrac 1 2\right]\mathclose{}/\Z, \quad [y] \longmapsto \frac{1}{2^{k+1}} (\psi_1(z,z) + \psi_0(dz,z))\]
where $dz = 2^ky$; we will sketch a more invariant description of this form below. Ranicki shows that this quadratic linking form is non-degenerate and its class in the Witt group of quadratic linking forms \cite[page 271]{phonebook} only depends on the class of $C$ in $\L_{0}^{q}(\Z,2)$; in fact in (the proof of) \cite[Proposition 3.4.7]{phonebook} it is shown that this group is isomorphic to the Witt group of linking forms via the above construction.

We claim that $\psi_1 = 0$ for $C = E \otimes F$. Note that this is true for $F$ by construction and persists to the product.
It is thus a simple calculation that the linking form $q$ associated to $E \otimes F$ is 
\[(\Z/2)^2 \longrightarrow \Z\mathopen{}\left[\tfrac 1 2\right]\mathclose{}/\Z, \quad (a,b) \longmapsto \frac{a^2+b^2+ab}{2}\]
Now, the Witt-group of quadratic $2$-primary linking forms over $\Z$ carries a $\Z/8$-valued Brown-Kervaire invariant $\beta$ defined via a Gau\ss-sum; see for example \cite[page 5]{taylorGauss}. It is easily checked straight from the definition that $\beta(q) = 4$, which finishes the proof that $\partial ef \neq 0$. 

In fact, \cite[Proposition 4.3.2]{phonebook} shows that $\L_0(\Z,2) \cong \Z/8 \oplus \Z/2$, the first factor detected via the Brown-Kervaire invariant, second factor detected by the $2$-adic logarithm of the size of the domain of a linking form.
\end{proof}

Finally, let us give a more conceptual description of the linking form employed in the proof of \cref{Prop_homotopyL}. As before, we will assign to a $1$-dimensional quadratic Poincar\'e chain complex $(C, \psi)$ over $\Z$ such that $C\mathopen{}\left[\tfrac 1 2\right]\mathclose{} \simeq 0$ a quadratic form 
\[
\mathrm H_{0}(C) \lto {\Z\mathopen{}\left[\tfrac 1 2\right]\mathclose{}}/{\Z}.
\]
To this end recall that $\psi$ lies in 
\[
\Hom\big(C \otimes_\Z C, \Z[1]\big)_{\mathrm{hC}_2} \simeq \Hom\big(C \otimes_\Z C, \Z\mathopen{}\left[\tfrac 1 2\right]\mathclose{}/\Z\big)_{\mathrm{hC}_2}  , 
\]
where the equivalence is obtained from the fibre sequence
\[
\Z\mathopen{}\left[\tfrac 1 2\right]\mathclose{} \lto \Z\mathopen{}\left[\tfrac 1 2\right]\mathclose{}/\Z \lto  \Z[1]
\]
because $C\mathopen{}\left[\tfrac 1 2\right]\mathclose{} \simeq 0$. Now, there is a canonical map 
\begin{equation*}\label{map_quad}
\Hom\big(C \otimes_\Z C, \Z\mathopen{}\left[\tfrac 1 2\right]\mathclose{}/\Z\big)_{\mathrm{hC}_2} \lto \Hom\big((C \otimes_\Z C)^{\mathrm{hC}_2}, \Z\mathopen{}\left[\tfrac 1 2\right]\mathclose{}/\Z\big).
\end{equation*}
In fact, using the quadraticity of both sides regarded as a functor of $C$, it is not difficult to check that this map is an equivalence for all perfect $C$, but we will not need this here. The element $\psi$ thus gives rise to a map $(C \otimes_\Z C)^{\mathrm{hC}_2} \to \Z\mathopen{}\left[\tfrac 1 2\right]\mathclose{}/\Z$ and in particular, and induced one 
\[
\mathrm H_{0}\big( (C \otimes_\Z C)^{\mathrm{hC}_2}\big) \lto \Z\mathopen{}\left[\tfrac 1 2\right]\mathclose{}/\Z \ .
\]
Now the quadratic form in question is obtained by  precomposing this group homomorphism with the quadratic form constructed in the following lemma for $M = C$ and $R = \Z$.

\begin{Lemma}\label{Lem_quadratic}
Let $M$ be a module over any $\E_2$-ring spectrum $R$. Then the map   
\[
q\colon \pi_0(M) \lto \pi_0\left( (M \otimes_R M)^{\mathrm{hC}_2}\right)
\]
obtained as $\pi_0$ of the composition
\[
\Omega^\infty M \stackrel{\Delta}{\lto} \left(\Omega^\infty M \times \Omega^\infty M \right)^{\mathrm{hC}_2} \lto 
\Omega^\infty \left( M \otimes_R M\right)^{\mathrm{hC}_2}
\]
is quadratic over $\pi_0(R)$. 
\end{Lemma}
\begin{proof}
We have to show that $q(r x) = r^2 q(x)$ for $r \in \pi_0(R)$ and that 
\[
\beta(x,y) = q(x + y) - q(x) - q(y)
\]
is bilinear. 

The map $\Omega^\infty M \to \Omega^\infty \left( M \otimes_R M\right)^{\mathrm{hC}_2}$ is natural in $M$. In particular for every $r \in \Omega^\infty R$ the morphism $l_r: M \to M$ obtained by left multiplication by $r$ the diagram 
\[\begin{tikzcd}
\Omega^\infty M \ar[r]\ar[d,"{\Omega^\infty l_r}"] &  \Omega^\infty \left( M \otimes_R M\right)^{hC_2} \ar[d,"{\Omega^\infty(l_r \otimes l_r)^{hC_2}}"] \\
\Omega^\infty M \ar[r] &  \Omega^\infty \left( M \otimes_R M\right)^{hC_2}
\end{tikzcd}\]
commutes. But by bilinearity the right vertical map is equivalent to left multiplication with $r^2$ on 
the $R$-module $(M \otimes_R M)^{hC_2}$. Upon applying $\pi_0$ this implies the equality $q(r x) = r^2 q(x)$.

Similarly, applying naturality for the fold map $a\colon M \oplus M \to M$ we find a commutative square
\[\begin{tikzcd}
\Omega^\infty (M \oplus M)  \ar[r]\ar[d,"{\Omega^\infty a}"] &  \Omega^\infty \big( (M \oplus M) \otimes_R (M \oplus M)\big)^{hC_2} \ar[d,"{\Omega^\infty(a \otimes a)^{hC_2}}"] \\
\Omega^\infty M \ar[r] &  \Omega^\infty \left( M \otimes_R M\right)^{hC_2} .
\end{tikzcd}\]
Under the distributivity equivalence  
\[
\big((M \oplus M) \otimes_R (M \oplus M)\big)^{hC_2} \simeq (M \otimes_R M)^{hC_2}  \oplus ( M \otimes_R M)  \oplus (M \otimes_R M)^{hC_2}
\]
the right vertical map $(a \otimes a)^{hC_2}$ in this square is given by $(\id,\mathrm{Nm}, \id)$ where $\mathrm{Nm}$ is the norm map $M \otimes_R M \to (M \otimes_R M)^{hC_2}$. Thus applying $\pi_0$ we get the identity
\[
q(x + y) = q(x) + (\pi_0 \mathrm{Nm})(x \otimes y) + q(y)
\]
or equivalently $\beta(x,y) = (\pi_0 \mathrm{Nm})(x \otimes y)$. But $\mathrm{Nm}$ is $R$-linear which implies the claim.
\end{proof}

\begin{Rmk}
%\begin{enumerate}
%\item 
Employing the map
\[\Hom_{R \otimes R}(C \otimes C, R[4n])_{\mathrm{hC}_2} \longrightarrow \Hom_{R \otimes R}((C[-2n] \otimes C[-2n])^{\mathrm{hC}_2}, R_{\mathrm{hC}_2})\]
\cref{Lem_quadratic} also provides the $\epsilon$-quadratic forms
\[\mathrm{H}_{2n}(C) \longrightarrow R/(r-\epsilon \bar{r} \mid r \in R)\] for a $4n$-dimensional $\epsilon$-quadratic Poincar\'e complex over a ring $R$ with involution $r \mapsto \bar{r}$, and similarly it generally provides the (split) $\epsilon$-quadratic linking form 
\[\mathrm H_{2n}(C) \longrightarrow R[S^{-1}]/(r + s -\epsilon \bar{s} \mid r \in R, s \in R[S^{-1}])\] associated with an $S$-acyclic $4n+1$-dimensional $\epsilon$-quadratic Poincar\'e complex, compare \cite[Section 3.4]{phonebook}.
%\item Evidently, for any abelian group $A$ the map 
%\[
%A \lto (A \otimes A)^{C_2} \qquad a \mapsto a \otimes a
%\]
%is quadratic; in fact, one can show with some work that this is the initial such quadratic morphism, i.e.\ that $(A \otimes A)^{C_2}$ corepresents the functor, which sends an abelian group $B$ to quadratic forms on $A$ with values in $B$.
%\end{enumerate}
\end{Rmk}

\bibliographystyle{amsalpha}
\bibliography{bibliography}

\end{document}